\documentclass[reqno,a4paper]{amsart}     

\usepackage{a4wide}
\usepackage{amsmath,amssymb,amsaddr}    
\usepackage{url}
\usepackage[numbers,sort]{natbib}
\usepackage{graphicx}

\usepackage[symbol]{footmisc}

\usepackage{enumerate}
\usepackage[english]{babel}

\theoremstyle{plain}
\newtheorem{thm}{Theorem}
\newtheorem{lem}[thm]{Lemma}
\newtheorem{claim}[thm]{Claim}
\newtheorem{observation}[thm]{Observation}

\newtheorem{rem}[thm]{Remark}
\newtheorem{defn}[thm]{Definition}
\newtheorem{fact}[thm]{Fact}

\newtheorem{remark}[thm]{Remark}
\theoremstyle{definition}

\newcommand{\eps}{\varepsilon}

\newcommand{\NN}{\mathbb{N}}
\newcommand{\ZZ}{\mathbb{Z}}
\newcommand{\RR}{\mathbb{R}}
\newcommand{\EE}{\mathbb{E}}

\newcommand{\cE}{\mathcal{E}}
\newcommand{\cF}{\mathcal{F}}

\newcommand{\cond}{\; \middle\vert \;}
\newcommand{\Poi}[1]{\mathrm{Po}\left(#1\right)}
\newcommand{\Bin}[2]{\mathrm{Bin}\left(#1,#2\right)}

\newcommand{\Space}{\mathbb{T}^d} 
\newcommand{\w}[1]{w_{#1}}
\newcommand{\x}[1]{x_{#1}}
\newcommand{\pld}{\mathcal{D}} 
\newcommand{\wmin}{w_{\text{min}}}
\newcommand{\V}[1]{{\ifx&#1& \operatorname{Vol}\else\operatorname{Vol}\!\left(#1\right)\fi}}
\newcommand{\infthr}{\rho_c}
\newcommand{\infpar}{\rho}
\newcommand{\VS}[4]{{\ifx&#4&V\else N\fi}_{#1}^{\ifx&#3&\else{#3}\fi}{\ifx&#4&\else(#4)\fi}{\ifx&#2&\else{\cap\, #2}\fi}}
\newcommand{\Nuupp}[1]{\widetilde{\nu}_{#1}}
\newcommand{\Bupp}[1]{\widetilde{B}_{#1}}
\newcommand{\Nulow}[1]{{\nu}_{#1}}
\newcommand{\Nulowtruncated}[1]{\kappa_{#1}}
\newcommand{\iCompleteSpace}{i_{\Space}}
\newcommand{\Blow}[1]{{B}_{#1}}
\newcommand{\Binit}{B_0}
\newcommand{\Nuinit}{\nu}
\newcommand{\act}[1]{L_{#1}}
\newcommand{\actExp}[1]{\Lambda{\ifx&#1&\else(#1)\fi}}
\newcommand{\whlow}[2]{w_{#1{\ifx&#2&\else{,#2}\fi}}}
\newcommand{\whupp}[2]{\tilde{w}_{#1{\ifx&#2&\else{,#2}\fi}}}

\renewcommand{\r}[1]{r_{#1}}

\newcommand{\eventB}{\mathcal{B}}
\newcommand{\eventC}{\mathcal{C}}
\newcommand{\eventD}{\mathcal{D}}
\newcommand{\eventE}[1]{\mathcal{E}(#1)}
\newcommand{\eventF}[2]{\mathcal{F}_{#1}(#2)}
\newcommand{\eventFProof}[1]{\mathcal{F}_*(#1)}
\newcommand{\eventG}[2]{\mathcal{G}_{#1}(#2)}
\newcommand{\eventH}[1]{\mathcal{H}(#1)}
\newcommand{\eventK}[1]{\mathcal{K}(#1)}
  
\newcommand{\edgesSet}{M}
\newcommand{\edgesIn}{\edgesSet_{*}}
\newcommand{\edgesOut}{\edgesSet^{*}}
\newcommand{\edgesBoundary}[1]{\edgesSet_{#1}}

\newcommand{\heavyVertices}[2]{U_{#1}^{\ifx&#2&\else{#2}\fi}}
\newcommand{\wslow}[1]{\hat{w}_{#1}}

\begin{document}
\title{Bootstrap percolation on geometric inhomogeneous random graphs}
\thanks{Most of the research was conducted during the first author's stay at ETH Z\"urich. The first author was/is supported by Austrian Science Fund (FWF): P26826 and W1230, European Research Council (ERC): No.~639046, and Engineering and Physical Sciences Research Council (EPSRC): No.~EP/N004833/1.\\
\indent\indent An extended abstract on the results in this manuscript appeared in the proceedings of the ``43rd International Colloquium on Automata, Languages and Programming (ICALP)''}

\author[C.~Koch]{Christoph Koch}
\address{Department of Statistics\\University of Oxford, Oxford OX1 3LB, United Kingdom}
\email{christoph.koch@stats.ox.ac.uk}
\author[J.~Lengler]{Johannes Lengler}
\address{Department of Computer Science\\ ETH Z\"urich, 8092 Z\"urich, Switzerland} \email{johannes.lengler@inf.ethz.ch}

\begin{abstract}
Geometric inhomogeneous random graphs (GIRGs) are a model for scale-free networks with underlying geometry. We study bootstrap percolation on these graphs, which is a process modelling the spread of an infection of vertices starting within a (small) local region. We show that the process exhibits a phase transition in terms of the initial infection rate in this region. We determine the speed of the process in the supercritical case, up to lower order terms, and show that its evolution is fundamentally influenced by the underlying geometry. For vertices with given position and expected degree, we determine the infection time up to lower order terms. Finally, we show how this knowledge can be used to contain the infection locally by removing relatively few edges from the graph. This is the first time that the role of geometry on bootstrap percolation is analysed mathematically for geometric scale-free networks.
 \end{abstract}

\maketitle

\section{Introduction}

One of the most challenging and intriguing questions about large real-world networks is how activity spreads through the network. ``Activity'' in this context can mean many things, including infections in a population, the dissemination of opinions and rumours in social networks, viruses in computer networks, action potentials in neural networks, and many more. While all these networks seem very different, in the last two decades there was growing evidence that most of them share fundamental properties~\cite{albert2002statistical, dorogovtsev2002evolution}. The most famous property is that the networks are \emph{scale-free}, i.e.\ the degree of a vertex $v$ follows a power-law distribution $\Pr[\deg(v) \geq d] \approx d^{1-\beta}$, typically for some $2<\beta<3$. Other properties include a large connected component which is a small world (poly-logarithmic diameter) and an ultra-small world (constant or poly-loglog average distance), that the networks have small separators and a large clustering coefficient. We refer the reader to~\cite{hofstadVol1,hofstadVol2} for an overview. 

Classical models for random graphs fail to have these common properties. For example, Erd\H{o}s-R\'enyi graphs or Watts-Strogatz graphs do not have power-law degrees, while Chung-Lu graphs and preferential attachment (PA) graphs fail to have large clustering coefficients or small separators. The latter properties typically arise in real-world networks from an underlying geometry, either spatial or more abstract, e.g., two nodes in a social networks might be considered ``close'' if they share similar professions or hobbies. Since the spread of activity (of the flu, of viral marketing, ...) in real-world networks  is crucially governed by the underlying spatial or abstract geometry~\cite{lambiotte2009communities}, the explanatory power of classical models is limited in this context.

In recent years models have been developed which overcome the previously mentioned limitations, most notably \emph{hyperbolic random graphs} (HypRGs)~\cite{boguna2010sustaining,boguna2009navigability,barabasi2012network,papadopoulos2012popularity} and their generalization\footnote{It is non-obvious that GIRGs are a generalisation of HypRGs, see~\cite[Theorem 2.3]{bringmann2018geometric}.} \emph{geometric inhomogeneous random graphs} (GIRGs)~\cite{bringmann2018geometric,bringmann2015generalGIRG,bringmann2017greedy,bringmann2017sampling,blasius2019efficiently}, \emph{scale-free percolation} (SFP)~\cite{deijfen2013scale,deprez2013scale,deprez2015inhomogeneous} and \emph{spatial preferential attachment} (SPA) models~\cite{aiello2008spatial,cooper2014some,jacob2015spatial}. Apart from the power-law exponent $\beta$, these models come with a second parameter $\alpha>1$, which models how strongly the edges are predicted by their distance. Due to their novelty, there are only very few theoretical results on how the geometry impacts the spreading of activity through these networks. 

In this paper we make a first step by analysing a specific process, \emph{bootstrap percolation}~\cite{chalupa1979bootstrap}, on the GIRG model~\cite{bringmann2018geometric}. In this process, an initial set of \emph{infected} (or \emph{active}) vertices iteratively infects all vertices which have at least $k$ infected neighbours, where $k\geq 2$ is a parameter. It was originally developed to model various physical phenomena (see~\cite{adler2003bootstrap} for a short review), but has by now also become an established model for the spreading of activity in networks, for example for the spreading of beliefs~\cite{granovetter1973strength,dreyer2009irreversible,shrestha2014message,ramos2015does}, behaviour~\cite{gersho1975simple,granovetter1978threshold}, or viral marketing~\cite{karbasi2015normalization} in social networks (see also~\cite{centola2010spread}), of contagion in economic networks~\cite{amini2013resilience}, of failures in physical networks of infrastructure~\cite{watts2002simple} or compute architecture~\cite{kirkpatrick2002percolation,flocchini2004dynamic}, of action potentials in neuronal networks (e.g,~\cite{schwenker1996iterative,tlusty2009remarks,amini2010bootstrap,cohen2010quorum,turova2012emergence,orlandi2013noise,einarsson2014bootstrap,einarsson2014high}, see also~\cite{Kozma2007neuropercolation} for a review), and of infections in populations~\cite{dreyer2009irreversible}. Bootstrap percolation has been intensively studied theoretically and experimentally on a multitude of models, including trees~\cite{balogh2006bootstrap}, lattices~\cite{aizenman1988metastability,balogh2012sharp}, Erd\H{o}s-R\'enyi graphs~\cite{janson2012bootstrap}, various geometric graphs~\cite{moukarzel2010long,bradonjic2014bootstrap,gao2014bootstrap}, and scale-free networks~\cite{dorogovtsev2006k,baxter2010bootstrap,amini2014bootstrap,karbasi2015normalization,fountoulakis2018}. On geometric scale-free networks there are some experimental results~\cite{carmi2007model}, but little is known theoretically.

While there is plenty of experimental literature and also some mean-field heuristics on other activity spreading processes on geometric scale-free networks (e.g.,~\cite{warren2002geography,xu2007impacts,huang2007epidemic,zhao2012epidemic,griffin2012community,schmeltzer2014percolation}), rigorous mathematical treatments are non-existent with the notable exception of  \cite{janssen2015rumours}, where rumour spreading is analysed in an SPA model with a push and a push\&pull protocol.

\subsection{Our contribution}
We investigate \emph{localised} bootstrap percolation on GIRGs with an expected number of $n$ vertices: given a ball $B$ in the underlying geometric space, we initially infect each vertex in the \emph{source region} $B$ independently with probability $\infpar$. In this way, we model that an infection (a rumour, an opinion, ...) often starts in some local region, and from there spreads to larger parts of the network. In Theorem~\ref{thm:main} we determine a threshold $\infthr$ such that in the \emph{supercritical case} $\infpar \gg \infthr$ whp\footnote{with high probability, i.e.\ with probability tending to $1$ as $n\to\infty$. } a linear fraction of the graph is infected eventually, and in the \emph{subcritical case} $\infpar \ll \infthr$ infection ceases immediately. In the \emph{critical case} $\infpar=\Theta(\infthr)$ both options occur with non-vanishing probability: if there are enough (at least $k$) ``local hubs'' in the source region, i.e.\ vertices of relatively large expected degree, then they become infected and facilitate the process. Without local hubs the initial infection is not dense enough, and comes to a halt. 

For the supercritical case, we show that it only takes $O(\log \log n)$ rounds until a constant fraction of all vertices is infected, and we determine the number of rounds until this happens up to a factor $1\pm o(1)$ in Theorem~\ref{thm:numberofrounds}.  For the matching lower bound in this result, we need the technical condition $\alpha > \beta-1$, i.e.\ edge-formation may not depend too weakly on the geometry. Notably, if the source region $B$ is sufficiently small then the number of rounds agrees (up to minor terms) with the average distance in the network. In particular, it does not depend on the infection rate $\infpar$, as long as $\infpar$ is supercritical.
  
Finally we demonstrate that the way the infection spreads is strongly determined by the geometry of the process, again under the assumption $\alpha > \beta-1$. Starting from $B$, the infection is carried most quickly by local hubs. Once the local hubs in a region are infected, they pass on their infection \emph{(i)} to other hubs that are even further away, and \emph{(ii)} locally to nodes of increasingly lower degree, until a constant fraction of all vertices in the region is infected. For a more detailed description of the infection dynamics, we refer to Figure~\ref{fig:bootstrap_percolation} on page~\pageref{fig:bootstrap_percolation} and to the (still informal) discussion in Section~\ref{sec:intuition}. Our analysis is quite precise: given a vertex~$v$ (i.e.\ given its expected degree and its distance from $B$), and assuming that $v$ is not too close to $B$, we can predict whp (Theorem~\ref{thm:inftime}) in which round it will become infected, again up to a factor $1\pm o(1)$. In real applications such knowledge is invaluable: for example, assume that a policy-maker only knows initial time and place of the infection, i.e.\ she knows the region $B$ and the current round $i$. In particular, she does not know $\infpar$, she does not know which edges are present, and she has no detailed knowledge about who is infected. Then we show that she is able to identify a region $B'$ in which the infection can be quarantined. In other words, by removing (from round $i$ onwards) all edges crossing the boundary of $B'$ whp the infection remains contained in $B'$. The number of edges to be deleted is relatively small: it can be much smaller than $n$ (in fact, any function $f(n) = \omega(1)$ can be an upper bound, if $i$ and $\V B$ are sufficiently small), and it is even much smaller than the number of edges \emph{inside} of $B'$. 

\subsection{Related work} 
The notion of localised bootstrap percolation relies heavily on a random graph model which has an underlying geometry. Previously, the only mathematical rigorous work in this context is due to Candellero and Fountoulakis~\cite{candellero2014bootstrap}, where they determined the threshold for bootstrap percolation on HypRGs (in the threshold case $\alpha=\infty$, cf.~below). However, they still assumed that the initial infection takes place \emph{globally}, i.e.\ whether any vertex is infected initially is independent of its position, and not \emph{locally} as in our paper, where no vertex outside of a certain geometric region is infected initially. This has two major consequences. 
\begin{enumerate}
\item[I.] In the global setting, the (expected) number of initially infected vertices needs to be polynomial in $n$ in order for the infection to start spreading significantly; while in our setting every ball containing an expected number of $\omega(1)$ vertices can initiate a large infection whp. 
\item[II.] Using our knowledge about how the process evolves in time with respect to the geometry, we show that the infection time of any vertex is mainly governed by its geometric position and its weight. On the other hand, with a global initial infection the infection times only depend on the expected degrees, which is non-geometric information encoded in the vertex weights of the GIRG.
\end{enumerate}

The GIRG model is closely related to the model of \emph{scale-free percolation} (SFP)~\cite{deijfen2013scale,deprez2013scale,deprez2015inhomogeneous,heydenreich2016structures}, where the vertex set is given by the infinite grid $\ZZ^d$. For both GIRGs and SFP the probability of a pair of vertices forming an edge is essentially given by the weights of its endpoints and their distance, and the presence of pairs of edges is independent of one another (conditional on the weights and positions of the vertices).  In fact, after rescaling GIRGs to contain an infinite number of vertices, and a transformation of the parameters, the edge probabilities in SFP fall into the class of functions that are covered by the GIRG model (see~\cite[Section~1.5]{Hofstad2017explosion}), with the major difference being that vertices are distributed randomly in $\RR^d$ for these (modified) GIRGs, and that the edge set in SFP contains a grid by definition. Recently, van der Hofstad and Komj\'athy~\cite{Hofstad2017explosion} studied SFP with additional edge weights and characterised the occurrence of  explosion phenomena by the distribution of the edge weights. They proved that the (weighted) distance of two uniformly chosen vertices converges in distribution to an a.s.\ bounded random variable, and Komj\'athy and Lodewijks~\cite{komjathy2018explosion} subsequently transferred the result to GIRGs and HypRGs. Komj\'athy, Lapinskas and Lengler~\cite{komjathy2020stopping} obtained similar results for the case that the edge weights depend on the vertex weights of their endpoints, for SFP, GIRGs, and HypRGs.

\section{Model and notation}
In this section we first define the random graph model that we will discuss in this paper, as introduced in~\cite{bringmann2018geometric}. We deviate slightly from~\cite{bringmann2018geometric} by generating the vertices by a Poisson point process; the model in~\cite{bringmann2018geometric} can be obtained by conditioning on the number of vertices being exactly $n$, as we explain below. 
Afterwards, we formally introduce localised bootstrap percolation. The last part of this section introduces some necessary notation and clarifies the use of asymptotic statements within the paper.
\subsection{Graph model}
A GIRG is a graph $G = (V,E)$ where both the vertex set $V$ and the edge set $E$ are random. 
Each vertex $v$ is represented by a pair $(\x{v},\w{v})$ consisting of a \emph{position} $\x{v}$ (in some \emph{ground space}) and a \emph{weight} $\w{v}\in \RR_{>0}$.
    
\par{\textbf{Ground space and positions.}} 
We fix a (constant) dimension $d \ge 1$ and consider the $d$-dimensional torus $\Space=\RR^d / \ZZ^d$ as the ground space. We usually think of it as the $d$-dimensional cube $[0,1]^d$ where opposite boundaries are identified and measure distances by the $\infty$-norm on $\Space$, i.e.\ for $x,y \in [0,1]^d$ we define 
\[
\|x-y\| := \max_{1 \le i \le d} \min\{|x_i-y_i|,1-|x_i-y_i|\}.
\] 

 The set of vertices and their positions are given by a homogeneous Poisson point process on $\Space$ with intensity $n \in \NN$.\footnote{Other than in~\cite{bringmann2018geometric} we do not condition on the number of vertices to be exactly $n$, which leads to slightly less technical proofs.} More formally, for any (Lebesgue-)measurable set $B\subseteq\Space$, let $\VS{}{B}{}{}$ denote (with slight abuse of notation) the set of vertices with positions in $B$. Then $|\VS{}{B}{}{}|$ is Poisson distributed with mean $n\V{B}$, i.e.\ for any integer $j\ge 0$ we have
\[
\Pr\left[|\VS{}{B}{}{}|=j\right]=\Pr[\Poi{n\V B}= j] = \frac{\left(n\V{B}\right)^j\exp(-n\V{B})}{j!},
\]
and if $B$ and $B'$ are disjoint measurable subsets of $\Space$ then $|\VS{}{B}{}{}|$ and $|\VS{}{B'}{}{}|$ are independent. 

Note in particular that the total number of vertices $|V|$ is Poisson distributed with mean $n$. An important property of this process is the following: Given a random vertex\footnote{By abuse of notation, $x_v$ and $w_v$ may either denote random variables or values.} $v=(\x{v},\w{v})$, if we condition on $x_v\in B$, where $B$ is some measurable subset of $[0,1]^d$, then the position $\x{v}$ is uniformly distributed in $B$. 

\par{\textbf{Weights.}} 
For each vertex, we draw independently a weight from some distribution $\pld$ on $\RR_{>0}$. We say that the weights follow a \emph{weak power-law} for some exponent $\beta\in(2,3)$ if a $\pld$-distributed random variable $D$ satisfies the following two conditions: there is a constant $\wmin \in \RR_{>0}$ such that $\Pr\left[ D \ge \wmin\right]=1$, and for every constant $\gamma>0$ there are constants $0<c_1\le c_2$ such that
\begin{equation}\label{eq:powerlaw}
c_1 w^{1-\beta-\gamma}\le \Pr\left[D\ge w\right]\le c_2 w^{1-\beta+\gamma}
\end{equation}
for all $w\geq \wmin$. If this condition is also satisfied for $\gamma =0$, then we say that the weights follow a \emph{strong power-law}.

\par{\textbf{Edges.}} 
Next we fix an $\alpha\in\RR_{>1}\cup\{\infty\}$. Then two distinct vertices $u=(\x{u},\w{u})$ and $v=(\x{v},\w{v})$ form an edge independently of all other pairs with probability $p(\x u, \x v, \w u, \w v)$. 

For $\alpha<\infty$ we assume that the function $p$ satisfies
\begin{equation}\label{eq:edgeprob1}
c_3\min\left\{\left(\frac{\w{u} \w{v}}{\|\x{u}-\x{v}\|^d n}\right)^\alpha,1\right\}\le p(\x u, \x v, \w u, \w v)\le c_4\min\left\{\left(\frac{\w{u} \w{v}}{\|\x{u}-\x{v}\|^d n}\right)^\alpha,1\right\},
\end{equation}
for some constants $0<c_3\le c_4$, and sufficiently large $n$. In the \emph{threshold model} $\alpha=\infty$ we instead require that $p$ satisfies 
\begin{equation}\label{eq:edgeprob2}
p(\x u, \x v, \w u, \w v)\begin{cases}
\ge c_7&\text{ if }\|\x u-\x v\|\le c_5\left(\frac{\w{u}\w{v}}{n}\right)^{1/d}\\
=0&\text{ if }\|\x u-\x v\|> c_6\left(\frac{\w{u}\w{v}}{n}\right)^{1/d}
\end{cases}
\end{equation}
for some constants $0<c_5\le c_6$,  and $c_7>0$, and sufficiently large $n$. Note that for $c_5\neq c_6$ the edge probability may be arbitrary in the interval $\left(c_5\left(\frac{\w{u}\w{v}}{n}\right)^{1/d},c_6\left(\frac{\w{u}\w{v}}{n}\right)^{1/d}\right)$. 

\subsection{Localised bootstrap percolation}
Let $k\ge2$ be a constant, let $\Binit\subseteq\Space$ be measurable, and let $0 \leq \infpar \leq 1$. Then (localised) bootstrap percolation with  \emph{threshold $k$}, \emph{source region $\Binit$}, and \emph{initial infection rate $\rho$} is the following process. For every integer $i\ge 0$ there is a set $\VS{}{}{\le i}{}\subseteq V$ of vertices which are \emph{infected} (or \emph{active}) at time $i$. The process starts with a random set $\VS{}{}{\le 0}{}\subseteq V$ which contains each vertex in $\VS{}{\Binit}{}{}$ independently with probability $\infpar$, and which contains no other vertices. Then we define iteratively
\[
\VS{}{}{\le i+1}{} := \VS{}{}{\le i}{} \cup \left\{v \in V \cond \text{ $v$ has at least $k$ neighbours in $\VS{}{}{\le i}{}$}\right\}
\]
for all $i\geq 0$. 
Moreover, we set $ \VS{}{}{\le \infty}{} := \bigcup_{i=0}^{\infty} \VS{}{}{\le i}{}$. For a vertex $v\in V$, we define its \emph{infection time} as $\act{v}:=\inf\left\{i\ge0\cond v\in \VS{}{}{\le i}{}\right\}$ and $\act{v}:=\infty$ if the infimum does not exist.

We denote by $\Nuinit=\Nuinit(n) := n\V{\Binit}$ the expected number of vertices in $\Binit$. Throughout the paper we will assume that $\Binit$ is a closed ball (with respect to $\|\cdot\|$), which is -- without loss of generality due to symmetry of $\Space$ -- centered at $0$.

\subsection{Asymptotic expressions and further notation}\label{sec:notation}

In general we will be interested in results for large values of $n$ (the expected number of vertices), and in particular we use the phrase \emph{with high probability} (whp) to mean with probability tending to $1$ as $n\to\infty$. Moreover, all unspecified limits and asymptotics will be with respect to $n\to\infty$, and whenever we say that a quantity is a constant, this means that it is independent of the parameter $n$. Furthermore, all constants hidden by Landau-notation are positive: for example, for a function $f=f(n)$ the notation $f = O(1)$ means that there is $n_0 >0$ and a constant $C>0$ that depends only the constant parameters $\alpha,\beta,d,\wmin$, and $k$ of the model, and on the implicit constants $c_1,c_2,\ldots,c_7$ in the definition of $\pld$ and $p$, such that $f(n) \leq C$ for all $n \geq n_0$. Similarly, $f = \omega(1)$ means $\lim_{n\to\infty} f(n) = \infty$ etc. We also combine the notions of whp and Landau notation: for instance, whp we have $f(n)=O(1)$ means that there is $n_0>0$ and a constant $C>0$ (as above) such that for every $\delta>0$ we have $\Pr[f(n)\le C]\ge 1-\delta$ for all $n\ge n_0$.

In the proofs, for the sake of readability, we will not state each time when we use inequalities that only hold for sufficiently large $n$. For example, we will in general assume that $\Nuinit=\omega(1)$ and thus we will use inequalities like $\Nuinit >2$ without further comment although they are only true for sufficiently large $n$. 

\par{\textbf{Further notation.}}
Throughout the paper, whenever we consider some ball $B\subset\Space$, it will be a closed ball with respect to the norm $\|\cdot\|$. In particular, the volume of a ball of radius $0\le r\le 1/2$ is precisely $(2r)^d$. For any $\lambda \geq 0$ and any closed ball $B\subseteq \Space$ of radius $r\geq  0$ centered at $0$ we denote by $\lambda B$ the closed ball of radius $\lambda r$ around $0$; in case $\lambda r\ge 1/2$ this yields the entire ground space, i.e.\ we have $\lambda B=\Space$. For any two sets of vertices $U_1$ and $U_2$, we denote the set of edges between them by $E\left(U_1,U_2\right):=\left\{e=\{u_1,u_2\}\cond u_1\in U_1,u_2\in U_2\right\}$. 

Throughout the paper we will ignore all events of probability $0$. For example, we will always assume that $V$ is a finite set, and that all vertices in $V$ have different positions. Furthermore, whenever it does not affect the argument, we omit floors and ceilings.

We will often make statements about vertices $v = (\x v, \w v)$ with fixed position and weight. This means that we condition the Poisson point process on having a vertex at position $\x v$, i.e, we are considering the corresponding Palm distribution. Then the remaining vertex set follows the same distribution as given by the original Poisson point process~\cite{daley2007introduction}. Moreover, since all vertex weights are drawn independently, for any fixed subset of the vertices we may condition on their weights. The resulting probability space is then given by the remaining random choices, i.e., the distribution of the remaining vertices and their weights, and of all edges in the graph.

\section{Main results}
The goal of this paper is to analyse the evolution of a localised bootstrap percolation on GIRGs as the expected number of vertices $n$ tends to $\infty$. First of all we show that localised bootstrap percolation on a GIRG has a threshold with respect to the initial infection rate $\infpar$. 
\begin{thm}\label{thm:main}
Let $G = (V,E)$ be a GIRG whose vertex weights follow a power-law with exponent $\beta\in(2,3)$ and consider a localised bootstrap percolation process on $G$ with initial infection rate $\rho = \rho(n) \in [0,1]$ and source region $B_0$ satisfying $\Nuinit=\Nuinit(n):=n\V{\Binit} = \omega(1)$. 
Then the critical infection rate $\infthr$ is given by 
\[
\infthr=\infthr(n):=\Nuinit^{-\frac{1}{\beta-1}},
\]
in the following sense:

If the weights follow a strong power-law, then:
\begin{enumerate}[(i)]
\item If $\infpar = \omega(\infthr)$, then $|\VS{}{}{\le\infty}{}|=\Theta(n)$ whp.
\item If $\infpar = \Theta(\infthr)$, then $|\VS{}{}{\le\infty}{}|=\Theta(n)$ with probability $\Omega(1)$, but also $\VS{}{}{\le\infty}{}=\VS{}{}{\le0}{}$ with probability $\Omega(1)$.
\item If $\infpar=o(\infthr)$, then $\VS{}{}{\le\infty}{}=\VS{}{}{\le0}{}$ whp.
\end{enumerate}\smallskip

If the weights follow a weak power-law, then:
\begin{enumerate}[(i)]
\item[(iv)] If there exists a constant $\epsilon_0 >0$ such that $\infpar\ge \infthr^{1-\epsilon_0}$, then $|\VS{}{}{\le\infty}{}|=\Theta(n)$ whp.
\item[(v)] If there exists a constant $\epsilon_0 >0$ such that $\infpar\le \infthr^{1+\epsilon_0}$, then $\VS{}{}{\le\infty}{}=\VS{}{}{\le0}{}$ whp.
\end{enumerate}
\end{thm}
We note that for global initial infections, i.e.\ when $\Nuinit=n$, this threshold agrees with the critical infection rate determined in~\cite{candellero2014bootstrap} on (threshold) hyperbolic random graphs. This is not surprising, as hyperbolic random graphs are a special instance of GIRGs~\cite{bringmann2018geometric}. However, in~\cite{bringmann2018geometric} it was only shown that hyperbolic random graphs satisfy a weak power law, so the results in~\cite{candellero2014bootstrap} in their full strength are not a formal consequence of our results.

Whenever we refer to the \emph{supercritical} regime we mean case~(i) and~(iv). Similarly, cases~(iii) and~(v) form the \emph{subcritical} regime and~(ii)  is the \emph{critical} regime. Note in particular that there is a supercritical regime regardless of how small the expected number $\Nuinit$ of vertices in the source region is, provided that $\Nuinit = \omega(1)$. This is in sharp contrast to non-geometric graphs like Chung-Lu graphs, where the expected number of vertices being infected initially must be polynomial in $n$ (if the initial infection is chosen at random).

Indeed the proof of Theorem~\ref{thm:main} will grant a deeper insight into the evolution of the process. Since the process whp stops immediately in the subcritical regime, we may restrict ourselves to the other cases. We show a doubly logarithmic upper bound on the number of rounds until a constant fraction of all vertices are infected. Furthermore, we prove that this bound is tight up to minor order terms if the influence of the underlying geometry on the random graphs is sufficiently strong, more precisely, as long as $\alpha>\beta-1$. 

\begin{thm}\label{thm:numberofrounds} 
Assume that we are in the situation of Theorem~\ref{thm:main}, let $\delta>0$ be constant and set 
\[
i_\infty := \frac{\log \log_\Nuinit n + \log\log n}{|\log (\beta-2)|}.
\]

Then in the supercritical regime whp, and in the critical regime with probability $\Omega(1)$, we have
$$
|\VS{}{}{\le(1+\delta)i_\infty}{}|=\Theta(n).
$$

If furthermore $\alpha > \beta-1$ and there exists a constant $C>\frac{\beta-1}{\beta-2}$ such that $\Nuinit^C \le n$, then in all regimes whp we have 
$$
|\VS{}{}{\le(1-\delta)i_\infty}{}|=o(n).
$$ 
\end{thm}

 Remarkably, the bounds do not depend on the initial infection rate $\rho$, as long as $\rho$ is supercritical. Moreover, if the expected number $\Nuinit$ of vertices in the source region is sufficiently small (if $\log \log \Nuinit= o(\log \log n)$), then $i_\infty=(2-o(1))\log\log n/|\log (\beta-2)|$ and thus $i_\infty$ coincides with the average distance of the GIRG which was determined in~\cite{bringmann2015generalGIRG} in a much more general setup, and also with the time that a greedy routing algorithm takes~\cite{bringmann2017greedy}. The proof of Theorem~\ref{thm:main} and Theorem~\ref{thm:numberofrounds} can be found in Section \ref{sec:proofmain}.

In fact, we can still refine the statement of Theorem~\ref{thm:numberofrounds} to obtain statements for individual vertices. In the following, we prove that for a vertex $v=(\x{v},\w{v})$ far enough from the origin, its infection time $\act{v}$ is determined as a function of  $(2\|\x{v}\|)^dn$ (i.e.\ the expected number of vertices in a ball of radius $\|\x{v}\|$) and its weight $\w{v}$. 

More precisely, for any $\x{}\in \Space\setminus \Binit$ and $\w{}\in\RR_{>0}$ we define\footnote{Using the convention that $\log y=-\infty$ for all $y\le 0$. Hence, $\actExp{\x{},\w{}}=0$ if $(2\|\x{}\|)^dn/ \w{} \le 1$.}
\begin{equation}\label{eq:defofellv}
\actExp{\x{},\w{}}:=\begin{cases}
\max\left\{0,\ \frac{\log\log_{\Nuinit}\left((2\|\x{}\|)^dn/\w{}\right)}{|\log(\beta-2)|}\right\}, & \text{ if } \w{}> ((2\|\x{}\|)^dn)^{1/(\beta-1)},\\
\\
\frac{2\log\log_{\Nuinit} ((2\|\x{}\|)^dn)-\log\log_{\Nuinit} \w{}}{|\log(\beta-2)|}, & \text{ if } \w{}\leq ((2\|\x{}\|)^dn)^{1/(\beta-1)}.
\end{cases}
\end{equation}
Observe that in the second case the sign of $\log\log_{\Nuinit} \w{}$ may be either positive or negative. However, we still have $\actExp{\x{},\w{}}=\Omega(1)$, since due to the upper bound on $\w{}$ we have 
$$
\actExp{\x{},\w{}} \geq \left[\log\log_{\Nuinit} ((2\|\x{}\|)^dn)+\log (\beta-1)\right]/|\log(\beta-2)|\ge \log (\beta-1)/|\log(\beta-2)|>0,
$$  
as $\x{}\in\Space\setminus\Binit$. We note that a few details of the definition in~\eqref{eq:defofellv} will be discussed in Section~\ref{sec:remarks}. 

The next result states that under some mild additional assumptions the infection time $\act{v}$ of a vertex $v=(\x{v},\w{v})$ outside $\Binit$ is given by $\actExp{\x{v},\w{v}}$ up to minor order terms.
\begin{thm}\label{thm:inftime}
Assume we are in the situation of Theorem~\ref{thm:main} in the supercritical regime and fix a constant $0<\eps<\frac{3-\beta}{\beta-2}$. Let $v=(\x{v},\w{v})$ be any fixed vertex such that $\x{v}\in \Space\setminus \Binit$ and $\w v = \omega(1)$. Then whp we have 
\begin{equation*}
\act{v} \leq (1+ o(1)) \actExp{\x{v},\w{v}} + O(1).
\end{equation*}
If additionally $\alpha>\beta-1$  and $\actExp{\x{v},\w{v}} \leq \tfrac{1}{d}\log_2\left((2\|\x{v}\|)^dn\Nuinit^{-(\beta-1)/(\beta-2)-\eps}\right)$ then whp also
\begin{equation*}
\act{v} \geq (1-o(1)) \actExp{\x{v},\w{v}} - O(1).
\end{equation*}
\end{thm}

As in Theorem~\ref{thm:numberofrounds}, the bounds do not depend on the initial infection rate $\rho$, as long as it is supercritical. In  Section~\ref{sec:remarks} we briefly discuss the assumption $\w v = \omega(1)$ in Theorem~\ref{thm:inftime}, as well as the necessity of  $\actExp{\x{v},\w{v}} \leq \tfrac{1}{d}\log_2\left((2\|\x{v}\|)^dn\Nuinit^{-(\beta-1)/(\beta-2)-\eps}\right)$ as an additional assumption for the lower bound.

Finally, we give a strategy how to contain the infection within a certain region when only the source region and the current round are known, but not the set of infected vertices. 

\begin{thm}\label{thm:containment}
Assume that we are in the situation of Theorem~\ref{thm:main}, and that $\alpha > \beta-1$. If the source region $\Binit$ is known, then for each integer $i\ge 0$ there exists a region $\Bupp{i}$ such that by removing all edges crossing the boundary of $\Bupp i$ before round $i+1$, whp the infection is contained in $\Bupp i$. Furthermore, for all constants $C>\frac{\beta-1}{\beta-2}$ and $c>\frac{1}{\beta-2}$ we may choose the regions $\Bupp{i}$ such that for all $i\ge 0$  we have $n\V{}(\Bupp{i})\le \Nuinit^{C c^i}$ and furthermore, the expected number of edges crossing the boundary of $\Bupp i$ is at most 
$$
(n\V{}(\Bupp{i}))^{\max\{3-\beta, 1-1/d\}+o(1)}.
$$
\end{thm}
Note that the number of edges that need to be removed is substantially smaller than the expected number of vertices $n\V{}(\Bupp{i})$ in the containment area $\Bupp i$. The proof of Theorem~\ref{thm:containment} can be found in Section~\ref{sec:containment}, while we formally define the regions $\Bupp{i}$ in Definition~\ref{def:BallsEtc} at the beginning of Section~\ref{sec:evolution}.

\subsection{Outline}\label{sec:intuition}

In this section we give an intuitive description of how the process evolves, and at the same time a very rough outline of the proofs. An illustration of the process is depicted in Figure~\ref{fig:bootstrap_percolation}. For simplicity, we restrict ourselves to the case of a strong power law. We warn the reader that some statements in this section are not literally true, but they are only true if appropriate error margins are taken into account. This holds in particular for the definition of the balls $B_i$ the quantities $\nu_i$, and the weights that will appear in the course of the argument. The precise definitions and exact statements are rather technical and are given in Section~\ref{sec:evolution}; the key statements are  Theorem~\ref{thm:speedlower} on page~\pageref{thm:speedlower} and Theorem~\ref{thm:speedupper} on page~\pageref{thm:speedupper}.

In Section~\ref{sec:preliminaries} we introduce some additional notation, collect some useful tools, and establish a number of basic properties of GIRGs, which will be relevant to the proofs. As already mentioned, Section~\ref{sec:evolution} is dedicated to the analysis of the evolution of the process and thus contains the heart of the proofs in full detail, as well as the proof of Theorem~\ref{thm:containment}. Section~\ref{sec:inftime} contains the derivation of Theorem~\ref{thm:inftime} whereas Section~\ref{sec:proofmain} deduces Theorems~\ref{thm:main} and~\ref{thm:numberofrounds} based on the key results of Section~\ref{sec:evolution}.

\par{\textbf{Bottleneck.}} We first discuss the very beginning of the process, i.e.\ the threshold behaviour as described by Theorem~\ref{thm:main}. For the subcritical regime, we distinguish between high-weight vertices ($\w v = \omega(w_0)$, where $w_0 = \nu^{1/(\beta-1)}$) and low-weight vertices ($\w v =O(w_0)$). By an easy computation, the expected number of low-weight vertices in $B_0$ that are infected in round $1$ is $o(1)$, so by Markov's inequality no low-weight vertex becomes infected whp. On the other hand, whp no high-weight vertex exists in $B_0$, and the expected number of infected vertices outside of $B_0$ is also $o(1)$ because they are too far away from infected vertices. In order words, whp no vertex is infected in round $1$.

In the critical regime, the calculation is similar, but if there exist vertices of weight $\Theta(w_0)$ then these vertices are infected with probability $\Omega(1)$. The number of vertices of weight $\Theta(w_0)$ is Poisson distributed with mean $\Theta(1)$, so it may happen (both with probability $\Omega(1)$) that either no such vertex exists (so percolation stops) or that there are at least $k$ such vertices, and that all of them are infected. In the supercritical regime, whp $k$ vertices of weight (slightly less than) $w_0$ are infected. Whp, these $k$ vertices infect all other vertices of similar weight in at most two more rounds. This is sufficient to start an \emph{avalanche of infection}, and for the rest of this section we will restrict ourselves to the case where this happens. We remark that the only relevant question for this stage is whether at least $k$ vertices of sufficiently high weight are infected. This would also allow to extend some of our results to other initial setting, e.g., if the infection probability $\rho$ is weight-dependent, similar to the weight-dependent transmission costs in~\cite{komjathy2020stopping}. However, we do not pursue these options further in this paper.

\begin{figure}[t]\begin{center}
		\includegraphics[width=0.95\columnwidth]{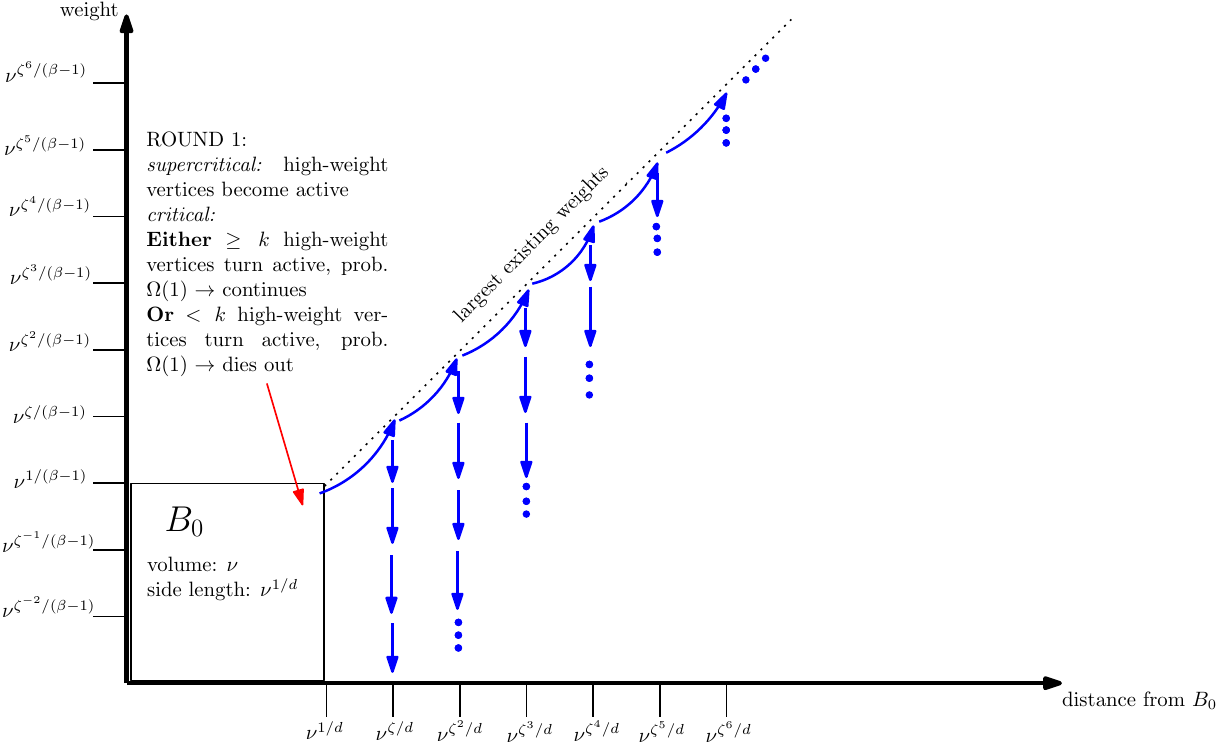}
		\caption{{\small Schematic depiction of the infection process after 7 rounds. In the first round, high-weight vertices in the source region $B_0$ are infected. In the $i$-th round after that, the high-weight vertices of $B_i$ (volume $\nu_i = \nu^{\zeta^i}$, distance $\nu_i^{1/d}$ from $B_0$) are infected by the largest-weight vertices in $B_{i-1}$. In the next round, the high-weight vertices in $B_i$ (weight $\nu_i^{1/(\beta-1)}$) infect high-weight vertices in $B_{i+1}$, but they also infect local vertices of slightly lower weight, namely vertices in $B_i$ of weight $\nu_i^{\zeta^{-1}/(\beta-1)} = \nu^{\zeta^{i-1}/(\beta-1)}$. In the next round, those in turn infect vertices of even lower weight $\nu_i^{\zeta^{-2}/(\beta-1)} = \nu^{\zeta^{i-2}/(\beta-1)}$ in $B_i$, and so on, until vertices of constant weights in $B_i$ are infected. So vertices in $B_i$ of weight $\nu^{\zeta^{i-\ell}}/(\beta-1)$ get infected in round $1+i+\ell$. Note that this description is schematic and simplified, see Theorem~\ref{thm:speedlower} and~\ref{thm:speedupper} for precise statements.}}

		\label{fig:bootstrap_percolation}
\end{center}\end{figure}

\par{\textbf{Growth of the infection region.}} If the infection gets started, then it evolves as follows. Let $\zeta = 1/(\beta-2) >1$ and consider the sequence $B_i$ of nested balls of volume $\nu_i/n$ centered at $0$, where $\nu_i \approx \nu^{\zeta^i}$. Then in the $i$-th round, all vertices of weight $w_i \approx \nu_i^{1/(\beta-1)}$ in $B_i$ are infected. In the next round, whp the vertices of weight $w_i$ in $B_i$ infect all vertices of weight $w_{i+1}$ in $B_{i+1}$, thus spreading the infection to new regions. Note that this statement is easy to prove inductively (cf.~Section~\ref{sec:prooflower}) since we assumed that \emph{all} vertices of weight $w_i$ in $B_i$ are infected, so for the vertices in $B_{i+1}$ it suffices to count the number of neighbours of a certain weight in $B_i$, which is a Poisson distributed random variable (cf.\ Fact~\ref{fact:Poisson}). This gives a lower bound on how fast the infection spreads geometrically. It can not spread faster since whp there are no edges from $B_i$ to $\Space \setminus B_{i+1}$. This latter fact already allows us to execute a quarantine strategy (Theorem~\ref{thm:containment}).

\par{\textbf{Infecting vertices of lower weight.}} If in round $j$ every vertex of weight $w$ in some region has a large probability to be infected, then in round $j+1$ every vertex of weight at least $w' \approx w^{1/\zeta}$ in this region has a large (though slightly smaller) probability to be infected. To prove this formally, we consider a vertex of weight $w'$. Such a vertex (but not vertices of smaller weight) has at least $w^{\delta}$ neighbours of weight $w$, with probability at least $1-\exp[-w^{\delta}]$. So we pick $k$ such neighbours, and bound the probability that at least one of them is \emph{not} infected by a union bound. In this way, we lose a factor of $k$ in each round, but going through the proof details it turns out that this factor is still negligible compared to the error term $\exp[-w^{\delta}]$. The full proof is contained in Section~\ref{sec:prooflower}. 

It is the most challenging and technical part of the proof to complement this infection pathway by a matching upper bound, which we do in Section~\ref{sec:proofupper}. Since in round $i-1$ there is no infected vertex in $B_i$ it is not hard to argue that in round $i$ only vertices of large weight in $\Space \setminus B_{i-1}$ are infected. However, in subsequent rounds it does happen that vertices of very small weight in $\Space \setminus B_{i-1}$ become infected. Fortunately, this only happens with rather small probability, which we can explicitly bound (Theorem~\ref{thm:speedupper} (f)) as a function of the weight. Once we have such a bound in some round, we use that whp no vertex in $\Space \setminus B_{i-1}$ (not too close to the boundary) has strictly more than one neighbour in $B_{i-1}$. Therefore, in order for a vertex $v$ in $\Space \setminus B_{i-1}$ to be infected, at least one of its neighbours in $\Space \setminus B_{i-1}$ must have been infected in the previous round, and we can bound the probability of this event by the expected number of previously infected neighbours in $\Space \setminus B_{i-1}$. It turns out that this simple bound is sufficient to provide the desired matching upper bound, safe quite some technical details for which we refer the reader to Section~\ref{sec:proofupper}. 

We remark that it is in this last step where we use the assumption $\alpha > \beta-1$ in the Theorems~\ref{thm:numberofrounds} and \ref{thm:inftime}, since otherwise there do exist vertices in $\Space\setminus B_{i-1}$ that have several neighbours in $B_{i-1}$, and these vertices exist in a substantial part of $B_{i}$. Even worse, for $\alpha <\beta -1$, in some (large) subregion of $B_i$ the number of infections in round $i+1$ that come from neighbours in $B_{i-1}$ dominates the number of infections that come from neighbours in $B_i$. For investigating the case $\alpha < \beta-1$ (which we don't in this paper), it will no longer be possible to use a bound on the infection probability that is uniform within $\Space\setminus B_{i-1}$, or within $B_i\setminus B_{i-1}$.

\par{\textbf{Infection times.}} Once the claims outlined above are proven (or rather their precise counterparts Theorem~\ref{thm:speedlower} and~\ref{thm:speedupper}) we have almost complete control over the process. In particular, for a each vertex $v$ with fixed weight and position (outside of the source region $B_0$), and for each round $j$ we have lower and upper bounds for the probability that $v$ is infected before round $j$. We can thus compute rounds $j_1,j_2$ for which the probability is at most $o(1)$ and at least $1-o(1)$, respectively, and we find that these rounds coincide up to lower order terms. It is still rather complicated to actually perform the calculations of $j_1$ and $j_2$ due to many technical details which we omitted in this outline, but no further knowledge about the infection process is required. Section~\ref{sec:inftime} contains the full proof.

\subsection{Additional remarks}\label{sec:remarks}
The first case in~\eqref{eq:defofellv} is not needed if we restrict ourselves to vertices as they typically appear in GIRGs. More precisely, as we will see in Lemma~\ref{lem:nolargeweights},~Section~\ref{sec:properties}, whp all vertices in $v=(\x{v},\w{v})\in \VS{}{(\Space\setminus \Binit)}{}{}$ satisfy $\w v \leq ((2\|\x v\|)^dn)^{1/(\beta-1-\lambda)}$ where $\lambda >0$ is an arbitrary constant. 
In the border case $((2\|\x v\|)^dn)^{1/(\beta-1)} \leq \w v \leq ((2\|\x v\|)^dn)^{1/(\beta-1-\lambda)}$ both expressions in~\eqref{eq:defofellv} agree up to additive constants, i.e.\
\begin{equation}\label{eq:defofellvsimplified}
\actExp{\x{v},\w{v}} = \frac{2\log\log_{\Nuinit} ((2\|\x v\|)^dn)-\log\log_{\Nuinit} \w{v}}{|\log(\beta-2)|} \pm O(1).
\end{equation}
Therefore, if we were to change Theorem~\ref{thm:inftime} so that it excludes vertices which are unlikely to exist, we could also use~\eqref{eq:defofellvsimplified} to define  $\actExp{}$.

The reason for the phase transition at $\w v \geq ((2\|\x v\|)^dn)^{1/(\beta-1)}$ is the very last step. A vertex with such a high weight has neighbours closer to the source region $B_0$ than all other vertices at distance $\|\x v\|$. For this reason it may turn active much earlier than any other vertex in its vicinity. Within this regime, a second phase transition occurs when $\w{v} = \|\x{}\|^dn/\nu^{\Theta(1)}$. At this point, $v$ starts to have direct neighbours inside of the source region $B_0$, and in particular it is connected to the vertices of largest weight in $B_0$, so $v$ becomes active after just two rounds. In~\eqref{eq:defofellv} this phase transition marks the point when the maximum in the first case is attained by zero.

Next we observe that the technical restrictions in Theorem~\ref{thm:inftime} are necessary: if a vertex has weight $\w v = O(1)$ then the number of neighbours is Poisson distributed with mean $\Theta(\w v)$  (see Lemma~\ref{lem:marginalProb} and Fact~\ref{fact:Poisson}), so $v$ is even isolated with probability $\Omega(1)$. In particular, we cannot expect that whp $v$ is ever infected. 

On the other hand, the restriction $\actExp{\x{v},\w{v}} \leq \tfrac{1}{d}\log_2\left((2\|\x{}\|)^dn\Nuinit^{-(\beta-1)/(\beta-2)-\eps}\right)$ ensures that $v$ is not too close to the source region. If $v$ is too close, then it may have neighbours inside of $\Binit$, and in this case it does depend on $\rho$ when they are infected. (And of course, this process iterates.) The term $\tfrac{1}{d}\log_2\left((2\|\x{}\|)^dn\Nuinit^{-(\beta-1)/(\beta-2)-\eps}\right)$ is not tight and could be improved at the cost of more technical proofs. However, there are already rather few vertices that violate the condition $\actExp{\x{v},\w{v}} \leq \tfrac{1}{d}\log_2\left((2\|\x{}\|)^dn\Nuinit^{-(\beta-1)/(\beta-2)-\eps}\right)$. For example, recall that it only takes $O(\log \log n)$ steps until a constant fraction of all vertices are infected. At this time, we only exclude vertices which satisfy $(2\|\x v\|)^dn \leq \Nuinit^{(\beta-1)/(\beta-2)+\eps} \cdot (\log n)^{O(1)}$, so the expected number of affected vertices is also at most $\Nuinit^{(\beta-1)/(\beta-2)+\eps} \cdot (\log n)^{O(1)}$, which is negligible if we assume $\Nuinit^C\le  n$ for some constant $C>\frac{\beta-1}{\beta-2}+\eps$. Even this is a gross overestimate, since the vertices close to the origin have much smaller infection times $\act{v}$, and thus only very few of them are affected by the condition.

\section{Preliminaries}\label{sec:preliminaries}
We often consider subsets of the vertex sets which satisfy some restrictions on their weights, positions, or whether they are infected at a given point of time. We use the following (slightly abusive) notation throughout the paper: For a weight $w\in \RR_{>0}$, a measurable set $B\subseteq \Space$, and a time $i\ge 0$ we set 
\[
\VS{\ge w}{B}{\le i}{}:=\left\{u=(\x{u},\w{u})\in V\cond \w{u}\ge w, \x{u}\in B, u\in \VS{}{}{\le i}{}\right\}.
\]
All three types of restrictions are optional. Moreover, we use the superscript ``$(=i)$'' to describe vertices which become infected (precisely) in round $i$, i.e. $\VS{}{}{=i}{}:=\VS{}{}{\le i}{}\setminus \VS{}{}{\le i-1}{}$ and $\VS{}{}{=0}{}:=\VS{}{}{\le 0}{}$ etc. Furthermore, the index ``$\ge w$'' may be replaced by ``$<w$'' or ``$\in [w,w')$'', with the obvious meaning. Additionally, we denote the neighbourhood of a vertex $v\in V$ by $\VS{}{}{}{v}:=\left\{u\in V\cond \{u,v\}\in E\right\}$ and this notation may be modified by the same three types of restrictions, i.e.\
\[
\VS{\ge w}{B}{\le i}{v}:= N(v) \cap (\VS{\ge w}{B}{\le i}{}).
\]

In~\cite{bringmann2018geometric}, GIRGs were defined with a fixed number of vertices, while we assume the set of vertices to be given by a homogeneous Poisson point process. One can retain the model in~\cite{bringmann2018geometric} by conditioning on the number of vertices being exactly $n$. Our choice allows for less technical proofs. In particular, one of the benefits of the Poisson point process is the following elementary fact.
\begin{fact} \label{fact:Poisson}
Let $\lambda\in\RR_{\ge0}$ and let $X$ be a Poisson distributed random variable with mean $\lambda$. Furthermore, given some $0\le q\le 1$, let $Y$ be a random variable which conditioned on $\{X=x\}$, for any $x\in \NN_0 = \{0,1,2,\ldots\}$, is the sum of $x$ independent Bernoulli random variables with mean $q$. Then $Y$ is Poisson distributed with mean $q\lambda$. 
\end{fact}
This means for instance that $|\VS{\geq w}{B}{}{v}|$ is Poisson distributed with mean $nq$, where $q$ denotes the probability that a vertex $u$ with random position $\x{u}$ and random weight $\w{u}$ satisfies $\w{u}\ge w$, $\x{u}\in B$, and is a neighbour of $v$.

We remark that the Poisson point process can be realised by a two-stage process: in the first stage, the \emph{number} of vertices is drawn from a Poisson distribution, and an ordering of the vertices is fixed. Then each vertex chooses independently a position in $\Space$ uniformly at random. Moreover, positions and weights of the vertices are independent, and we can draw them in any order. 

\subsection{Tools}
Many relevant quantities can be expressed by summing (some function) over all vertices whose weights lie in a given interval, the following lemma provides an easy way of evaluating these.
\begin{lem}\label{lem:integral}
Let $0 \leq w_0 < w_1$, and let $f:\RR_{\geq 0}\to\RR_{\geq 0}$ be a piecewise continuously differentiable function. Then for any 
finite set $V$ of weighted vertices, we have
\[
\sum_{v \in V, w_0 \leq \w{v} < w_1} f(\w{v}) \;=\; f(w_0) |V_{\geq w_0}| \;-\; f(w_1) |V_{\geq w_1}| \;+\; \int_{w_0}^{w_1} |V_{\geq w}| \frac{d}{dw}f(w) dw.
\]
In particular, if $f(0)=0$, $w_0=0$ and $w_1 > \max\{\w v \mid v\in V\}$, then we have 
\[
\sum_{v \in V} f(\w{v}) = \int_{0}^{w_1} |V_{\geq w}| \frac{d}{dw}f(w) dw=\int_{0}^{\infty} |V_{\geq w}| \frac{d}{dw}f(w) dw.
\]
\end{lem}
\begin{proof}
We will prove the lemma for a function $f$ which is everywhere continuously differentiable. The statement for \emph{piecewise} continuously differentiable functions then follows by applying this case to intervals $[w_i,w_{i+1}]$ on which $f$ is continuously differentiable, and summing over all these pieces. So assume that $f$ is everywhere continuously differentiable. Let $w_0 \leq t_1 < t_2 < \ldots <t_k < w_1$ be the distinct vertex weights $\w v$, where $v\in V$ with $w_0 \leq \w v \leq w_1$, sorted in non-decreasing order. Observe that
\begin{align*}
\sum_{v \in V, w_0 \leq \w{v} < w_1} f(\w{v}) &\;=\; \sum_{i=1}^n f(t_i)\left(\left|V_{\geq t_i}\right|-\left|V_{> t_i}\right|\right) \\
&\;=\; f(t_1)\left|V_{\geq t_1}\right| - f(t_n)\left|V_{\geq w_1}\right| + \sum_{i=2}^n f(t_i)\left|V_{\geq t_i}\right| - \sum_{i=2}^n f(t_{i-1})\left|V_{\geq t_i}\right| \\
&\;=\; f(t_1)\left|V_{\geq t_1}\right| - f(t_n)\left|V_{\geq w_1}\right| + \sum_{i=2}^n \left(f(t_i)-f(t_{i-1})\right)\left|V_{\geq t_i}\right|.
\end{align*}
For $2\leq i \leq k$, since $\left|V_{\geq x}\right| = \left|V_{\geq t_i}\right|$ for all $t_{i-1} < x \leq t_i$,
\begin{align*}
\int_{t_{i-1}}^{t_i}f'(x)\left|V_{\geq x}\right|dx = \left(f(t_i)-f(t_{i-1})\right)\left|V_{\geq t_i}\right|.
\end{align*}
Since $\left|V_{\geq x}\right| = \left|V_{\geq t_1}\right| = \left|V_{\geq w_0}\right|$ for all $w_0 < x \leq t_1$,
\begin{align*}
\int_{w_0}^{t_1}f'(x)\left|V_{\geq x}\right|dx = \left(f(t_1)-f(w_0)\right)\left|V_{\geq t_1}\right| = f(t_1)\left|V_{\geq t_1}\right| - f(w_0)\left|V_{\geq w_0}\right|.
\end{align*}
Similarly, $\int_{t_k}^{w_1}f'(x)\left|V_{\geq x}\right|dx = (f(w_1)-f(t_n))\left|V_{\geq w_1}\right|$. The claim follows by summing up the integrals.
\end{proof}

The next lemma spells out an almost trivial calculation, but since it is ubiquitous in our proofs, we state it as a lemma nevertheless. In our applications, $g(w)$ is typically the number of vertices of weight at least $w$ (possibly with additional restrictions), and $f$ is the probability that such a vertex has a certain property (e.g., that it forms an edge with some fixed $v$). After application of Lemma~\ref{lem:integral}, this almost always leads to an integral as given in~\eqref{eq:evalintegral} below.
\begin{lem}\label{lem:evalintegral}
Let $g:\RR_{\geq 0}\to\RR_{\geq 0}$ be a non-zero monomial, and let $f:\RR_{\geq 0}\to\RR_{\geq 0}$ be continuous and piecewise a non-zero monomial with non-negative exponent.\footnote{I.e., $g(w) = Cw^{r}$ for some $C>0$ and $r\in \RR$, and there exists a partitioning of $\RR_{\geq 0}$ into a finite number of intervals $I_1,\ldots,I_k$ such that for each $1\leq j \leq k$, the function $f$ restricted to $I_j$ is of the form $f(w) = C_jw^{s_j}$, for some $C_j >0$, $s_j \in \RR_{\geq 0}$. Moreover, $f$ is continuous.} Moreover, assume that there is $\tilde w$ such that the exponent of $w$ in $f(w)g(w)$ is strictly larger than $0$ for $w < \tilde w$, and strictly smaller than $0$ for $w > \tilde w$. Then for all $w_0,w_1$ such that $w_0 \leq \tilde w \leq w_1$, we have
\begin{equation}\label{eq:evalintegral}
\int_{w_0}^{w_1} g(w)\frac{d}{dw}f(w) dw = O(f(\tilde w)g(\tilde w)).
\end{equation}
Moreover, assume that \emph{(i)} the exponent of $f$ is non-zero in an interval $[(1-\Omega(1))\tilde w,\tilde w] \subseteq [w_0,\tilde w]$, \emph{or} \emph{(ii)} the exponent of $f$ is non-zero in an interval $[\tilde w, (1+\Omega(1))\tilde w] \subseteq [\tilde w, w_1]$. Then the $O(\cdot)$ in~\eqref{eq:evalintegral} may be replaced by $\Theta(\cdot)$.
\end{lem}
\begin{proof}
Let $g(w)= Cw^r$. Let us first assume that, by continuity, $f$ satisfies $f(w) = C_0w^{s_0}$ for $w \le \tilde w$ and $f(w) = C_1w^{s_1}$ for $w \ge \tilde w$, i.e.\ that $f$ consists of only two pieces. Then by assumption $r+s_0 > 0 > r+s_1 $. We first consider the lower part of the integral. If $s_0 =0$ then $(df/dw)(w) =0$ for $w\leq \tilde w$, and the integral from $w_0$ to $\tilde w$ vanishes. So assume that $s_0>0$. Then $(df/dw)(w) = C_0s_0 w^{s_0-1}$, and the antiderivative of $g\cdot (df/dw)$ is $CC_0s_0/(r+s_0) w^{r+s_0}$. Since $r+s_0 > 0$, this function is increasing in $w$, and
\begin{equation}\label{eq:evalintegral2}
\int_{w_0}^{\tilde w} g(w)\frac{d}{dw}f(w) dw = \frac{CC_0s_0}{r+s_0}\left(\tilde w^{r+s_0} -w_0^{r+s_0}\right) = \Theta(f(\tilde w)g(\tilde w) -f(w_0)g(w_0)).
\end{equation}
Note that if $w_0 \leq (1-\Omega(1))\tilde w$, then $f(\tilde w)g(\tilde w) -f(w_0)g(w_0) = \Omega(f(\tilde w)g(\tilde w))$ since $fg$ is a polynomial with positive exponent $r+s_0$ in $[w_0,\tilde w]$, which proves the additional statement~(i).

For the upper part of the integral, we may assume $s_1 >0$, since otherwise this part of the integral vanishes. Then $(df/dw)(w) = C_1s_1 w^{s_1-1}$, and the antiderivative of $g\cdot (df/dw)$ is $CC_1s_1/(r+s_1) w^{r+s_1}$. Note crucially that the sign of this function is negative since $r+s_1 <0$. Hence,
\begin{equation}\label{eq:evalintegral3}
\int_{\tilde w}^{w_1} g(w)\frac{d}{dw}f(w) dw = \frac{CC_1s_1}{-(r+s_1)}\left(\tilde w^{r+s_1}-w_1^{r+s_1}\right) = \Theta(f(\tilde w)g(\tilde w) -f(w_1)g(w_1)).
\end{equation}
Similarly to the first part, if $w_1 \geq (1+\Omega(1))\tilde w$, then $f(\tilde w)g(\tilde w) -f(w_1)g(w_1) = \Omega(f(\tilde w)g(\tilde w))$, which proves the additional statement (ii). On the other hand, Equation~\eqref{eq:evalintegral} follows immediately from~\eqref{eq:evalintegral2} and~\eqref{eq:evalintegral3} by leaving out the negative terms. This proves the lemma in the case that $f$ consists of only two pieces.

For the case of several pieces, the additional statement follows by restricting the integral to the two pieces bounded by $\tilde w$. For the upper piece, assume that $w_0 = w_0^{(1)} < \ldots < w_0^{(\lambda)} = \tilde w$ are the endpoints of the different pieces below $\tilde w$. Then in the same way as~\eqref{eq:evalintegral2}, we get
\begin{equation*}
\int_{w_0}^{\tilde w} g(w)\frac{d}{dw}f(w) dw = O(\sum_{i=1}^{\lambda} f(w_0^{(i)})g(w_0^{(i)})) = O(\lambda f(\tilde w)g(\tilde w)) = O(f(\tilde w)g(\tilde w)),
\end{equation*}
since $fg$ is an increasing function in $[w_0,\tilde w]$. The part $[\tilde w, w_1]$ follows analogously. 
\end{proof}

\subsection{Basic properties of GIRGs}\label{sec:properties}

In this section we list briefly some basic properties of GIRGs. We start with a fact which often allows us to treat the case $\alpha=\infty$ along with the case of finite $\alpha$ without case distinction.
\begin{observation}\label{obser:infinitealpha}
For every function $p$ satisfying~\eqref{eq:edgeprob2} and every $\alpha \in \RR_{>1}$, there is a function $\bar p$ satisfying~\eqref{eq:edgeprob1} such that $\bar p(x_1,x_2,w_1,w_2) \geq p(x_1,x_2,w_1,w_2)$ for all $x_1,x_2 \in \Space$ and all $w_1,w_2 \geq \wmin$.
\end{observation}
In other words, GIRGs in the threshold case $\alpha = \infty$ are dominated by GIRGs with finite $\alpha$. In particular, whenever we prove an upper bound on the number of active vertices that holds for all GIRGs with finite $\alpha$, the same upper bound also holds for threshold GIRGs.

The next lemma, taken from~\cite{bringmann2015generalGIRG}, tells us that the expected degree of a vertex equals its weight, up to constant factors. Moreover, it gives the marginal probability that two vertices $u,v$ of fixed weights but random positions in $\Space$ are adjacent. This probability remains the same if the position of one (but not both) of the vertices is fixed. 
\begin{lem}[Lemma~4.4 and Theorem~7.3 in~\cite{bringmann2015generalGIRG}]\label{lem:marginalProb}
Let $v=(\x{v},\w{v})$ be a vertex with fixed weight and position. Then
\begin{equation}\label{eq:marginal1}
\EE[\deg(v)] = \Theta(\w v).
\end{equation}
Moreover, if $u=(\x{u},\w{u})$ is a vertex with fixed weight, but with random position $\x{u}\in \Space$.\footnote{This means that we condition the Poisson point process on having at least one point, and for the first point in the ordered vertex set we condition on its weight, but not on its position, cf. the discussion at the beginning of Section~\ref{sec:preliminaries}.} Then 
\begin{equation}\label{eq:marginal}
\Pr\left[\{u,v\}\in E \mid \w u, \w v, \x v\right]=\Theta\left(\min\left\{\frac{\w{u} \w{v}}{n},1\right\}\right).
\end{equation}
Note in particular that the right hand side of~\eqref{eq:marginal} is independent of $\x v$, so the same formula still applies if also the position $\x v$ of $v$ is randomized.
\end{lem}
An expert reader may recognise that it is the same marginal probability as in Chung-Lu random graphs, cf.~\cite{bringmann2015generalGIRG} for a discussion in depth. 

Next we bound the expected number of neighbours with large weight of a fixed vertex. 
\begin{lem}\label{lem:weightofneighbours}
Let $\eta>0$ be a constant and consider a vertex $v=(\x{v},\w{v})$ with fixed weight and position. Then for every $w\ge\wmin$ we have 
\begin{enumerate}[\qquad(a)]
\item $\EE[|\VS{\ge w}{}{}{v}|] = O(\min\{\w v w^{2-\beta+\eta},nw^{1-\beta+\eta}\})$. \\
 In particular, for a random vertex $u$ we have, independently of $\x v$ and $\w v$,
 \[
 \Pr[\w u \geq w \mid \{u,v\}\in E] = O(w^{2-\beta+\eta});
 \]
\item $\EE[|\VS{\ge w}{}{}{v}|] = \Omega(\min\{\w v w^{2-\beta-\eta},nw^{1-\beta-\eta}\})$.
\end{enumerate}
\end{lem}
\begin{proof}
{\it (a)} By Lemma~\ref{lem:marginalProb}, the probability that a vertex $u$ with fixed weight $\w u$ and random position $\x{u}\in \Space$ is adjacent to $v$ is $\Theta(1)\min\{\w u\w v/n,1\}$. The expected number of vertices of weight at least $w$ is at most $O(n w^{1-\beta+\eta})$ by the power-law condition~\eqref{eq:powerlaw}. We distinguish two cases. If $w\w v \geq n$ then the probability to connect to any vertex of weight $w$ is $\Theta(1)$, so $\EE[|\VS{\ge w}{}{}{v}|] = \Theta(\EE[|\VS{\ge w}{}{}{}|])$, and the claim follows. So assume $w\w v \leq n$. Then by Lemma~\ref{lem:integral} we can compute the expectation as the following integral, which we then evaluate using Lemma~\ref{lem:evalintegral} with $\tilde w = n/\w v$.
\begin{align*}
\EE[|\VS{\ge w}{}{}{v}|] & = O\left(nw^{1-\beta+\eta}\frac{w\w v}{n} +\int_w^{\infty}n \w u^{1-\beta+\eta}\frac{d}{d\w u}\min\left\{\frac{\w v \w u}{n},1\right\} d\w u\right)\\
& = O\left(\w vw^{2-\beta+\eta}\right).
\end{align*}
We can write both cases uniformly as $\EE[|\VS{\ge w}{}{}{v}|] = O(\min\{\w v w^{2-\beta+\eta},nw^{1-\beta+\eta}\})$.

The second statement follows because the expected total number of neighbours of $v$ is $\Theta(\w v)$. Therefore, the probability that a random neighbour of $v$ has weight at least $w$ is $\Pr[\{u,v\}\in E\wedge\w u \geq w]/\Pr[\{u,v\}\in E] = O(w^{2-\beta-\eta})$, as required.

{\it (b)} This follows completely analogously to (a), except that we use that the expected number of vertices of weight at least $w$ is at least $\Omega(n w^{1-\beta-\eta})$ by the lower bound in the power-law condition~\eqref{eq:powerlaw}. 
\end{proof}

We often need to bound the expected number of neighbours of a given vertex in some geometric region, which we may do by the following lemma.
\begin{lem} \label{lem:neighExp}
Let $\eta>0$ and $C>1$ be constants, define $m:= \min\{\alpha,\beta-1-\eta\}$ and consider a closed ball $B \subseteq \Space$ of radius $r>0$ centered at $0$. Let $v = (\x v,\w v)$ be a vertex with fixed weight and position. Then
\[
\EE\left[|\VS{}{B}{}{v}|\right]=O(n\V B)\cdot
\begin{cases}
\min\left\{\frac{\w v}{n\V{B}}, 1\right\}, &\text{if } \|\x{v}\| \le Cr,\\ 
\min\left\{\left(\frac{\w{v}}{\|\x v\|^d n}\right)^m,1\right\}&\text{if } \|\x{v}\| \geq Cr.
\end{cases}
\]
\end{lem}
\begin{proof}
In the first case $\|\x{v}\| \le Cr$, the expected number of vertices in $B$ is $n\V{B}$, so clearly $\EE\left[|\VS{}{B}{}{v}|\right] \leq n\V{B}$. On the other hand, the expected number of neighbours of $v$ is $O(\w v)$, so $\EE\left[|\VS{}{B}{}{v}|\right] =O(1) \min\{\w{v},n\V{B}\}$.

For the second case, as before $\EE\left[|\VS{}{B}{}{v}|\right] \leq n\V{B}$. This proves the claim in the case $\w v \geq \|\x v\|^d n$, so assume otherwise. Observe that every vertex in $B$ has distance $\Theta(\|\x v\|)$ from $v$, and that the expected number of vertices in $B$ of weight at least $w$ is $O(n\V{B} w^{1-\beta+\eta})$. 

Consider first the case $\alpha < \beta-1-\eta$. Then by Lemma~\ref{lem:integral},
\begin{align*}
\EE\left[|\VS{}{B}{}{v}|\right]& = O\left(n\V{B}\min\left\{\left(\frac{\w v}{\|\x v\|^d n}\right)^\alpha,1\right\}  \right. \\ 
& \qquad \qquad \left. + \int_{\wmin}^{\infty} n\V{B} w^{1-\beta+\eta} \frac{d}{dw}\min\left\{\left(\frac{w \w v}{\|\x v\|^d n}\right)^\alpha,1\right\} dw\right).
\end{align*}
Note that the exponent of $w$ in the integrand is always negative, no matter which value the minimum attains. Moreover, recall that we assumed $\w v < \|\x v\|^d n$ and hence for $w = \wmin$ the minimum is $O((\w v/(\|\x v\|^d n))^\alpha)$. Thus by applying Lemma~\ref{lem:evalintegral} (with $\tilde w = \wmin = \Theta(1)$), the integral also evaluates to $O(n\V{B}\cdot \min\{(\w v/(\|\x v\|^d n))^\alpha,1\})$, as required. 

On the other hand, if $\alpha+1-\beta+\eta \geq 0$, then by Observation~\ref{obser:infinitealpha} we may restrict ourselves to $\alpha < \infty$, so we can estimate using Lemma~\ref{lem:integral} (with lower bound $0$)
\begin{align*}
\EE\left[|\VS{}{B}{}{v}|\right] & = O\left(\int_{0}^{\infty} n\V{B} w^{1-\beta+\eta} \frac{d}{dw}\min\left\{\left(\frac{w \w v}{\|\x v\|^d n}\right)^\alpha,1\right\} dw\right).
\end{align*}
 This integral evaluates to $O(n\V{B}\cdot (\w v/(\|\x v\|^d n))^{\beta-1-\eta})$, by Lemma~\ref{lem:evalintegral} (with $\tilde w = \|\x v\|^d n / \w v$). Since we have already shown that $\EE\left[|\VS{}{B}{}{v}|\right] \leq n\V B$, this proves the claim.
\end{proof}

In the last lemma of this section we show that whp there are no vertices whose weight is much larger than their distance from the origin. 

\begin{lem}\label{lem:nolargeweights}
Let $\eta>0$ be a constant and consider a closed ball $B\subseteq\Space$ centered at the origin $0$ satisfying $n\V{B}=\omega(1)$. Then with probability at least $1-(n\V{B})^{-\eta/4}$ there is no vertex $v=(\x{v},\w{v})$ with $\x v \in \Space\setminus B$ and $\w v \geq ((2\|\x v\|)^dn)^{1/(\beta-1-\eta)}$. 
\end{lem}
\begin{proof}
Let $\hat n$ be the number of such vertices, and denote the radius of $B$ by $0<r\le 1/2$. Let $r'>r$, then the probability density to find a vertex $v=(\x{v},\w{v})$ with $\|\x{v}\|=r'$ is equal to the volume of an $r'$-sphere\footnote{In $(\RR^d,\|\cdot\|_\infty$).} around $0$  that is intersected with $\Space$. By ignoring the intersection with $\Space$, we can only make the volume larger, so it is at most $O((2r')^{d-1}n)$. Moreover, the probability that a vertex has weight at least $w$ is at most $O(w^{1-\beta+\eta/2})$ by the power-law condition~\eqref{eq:powerlaw} (using $\gamma=\eta/2$). Hence, by Lemma~\ref{lem:integral} and Lemma~\ref{lem:evalintegral} we obtain
\begin{align*}
\EE\left[\hat n\right] & = O(1)\int_{r}^{\infty} (2r')^{d-1}n ((2r')^d n)^{(1-\beta+\eta/2)/(\beta-1-\eta)}dr' \le ((2r)^dn)^{-\eta/4},
\end{align*}
and the statement follows by Markov's inequality since $(2r)^d=\V{B}$.
\end{proof}

\section{Evolution of the process}\label{sec:evolution}

In this section we will prove two theorems which describe the geometrical evolution of the process in detail. First we show that in the supercritical regime the process will reach certain regions whp in a given time, yielding a \emph{lower bound on its speed}. This lower bound also applies in the critical regime if in the first step sufficiently many heavy vertices were activated, an event which holds with probability $\Omega(1)$.  Afterwards, we show that certain regions cannot be reached too early in the process, providing an \emph{upper bound on its speed}. From this we then derive Theorem~\ref{thm:containment} in Section~\ref{sec:containment}.

We start by defining two families of nested regions and a number of related parameters which will be crucial for describing the evolution of the process.
\begin{defn}\label{def:BallsEtc}
Set $\zeta := 1/(\beta-2)$ and note that $\zeta>1$. Moreover, let $0<\eps<\zeta-1$ be a constant and let $\eta=\eta(\eps)>0$ be a constant which is sufficiently small compared to $\eps$.\footnote{cf.~Remark~\ref{rem:epseta}} In the case of weak power laws, we additionally require that $\eta < (\beta-1)\epsilon_0/(1-\epsilon_0)$, where $\epsilon_0$ is the constant from Theorem~\ref{thm:main} (iv) and (v).
\begin{itemize}
\item For  all integers $i\geq 0$, we set 
\begin{align*}
\Nulow{0} & := \Nuinit & \text{ and } && \Nulow{i}  = \Nulow{i}(\eps) &:= \Nulow{0}^{(\zeta-\eps)^i}, \\
\Nuupp 0 = \Nuupp 0(\eps) & := \Nuinit^{(\beta-1)/(\beta-2)+\eps} & \text{ and } && \Nuupp i=\Nuupp i(\eps) &:= \Nuupp 0^{(\zeta+\eps)^i}.
\end{align*}
We then define $\Blow i = \Blow i(\eps)$ and $\Bupp i = \Bupp{i}(\eps)$ to be the closed ball centered around $0$ of volume $\V{\Blow{i}}:=\min\{\Nulow{i}(\eps)/n,1\}$ and $\V{}(\Bupp{i})=\min\{\Nuupp i(\eps)/n,1\}$, respectively. 
\item For all integers $i,\ell\ge 0$, we abbreviate $\Nulowtruncated{i}  :=\min\{\Nulow{i},n\}=n\V{\Blow{i}}$ and define 
\begin{align*}
\whlow{i}{\ell}  := \Nulowtruncated{i}^{(\zeta-\eps)^{-\ell}/(\beta-1+\eta)}\qquad \text{ and }\qquad\whlow{i}{}  :=\whlow{i}{0}.
\end{align*}
\item We denote by $\iCompleteSpace$ the smallest integer $i\ge 0$ such that $\Blow{i}=\Space$, i.e.\ we define
$$
\iCompleteSpace:=\min\{i\in\NN\colon \Nulow{i}\ge n\}.
$$
\end{itemize}
\end{defn}
First note that for all integers $i,\ell\ge 0$ the following two inequalities are satisfied
\begin{equation}\label{eq:Nulowtruncated}
\Nulowtruncated{i+1}\le\Nulowtruncated{i}^{\zeta-\eps} \qquad\text{ and } \qquad\whlow{i+1}{\ell} \le \whlow{i}{\ell}^{\zeta-\eps}.
\end{equation}
For $0\le i<\iCompleteSpace$, i.e.\ when $\Nulow{i}< n$, we have $\Nulowtruncated{i}=\Nulow{i}=\Nulow{i+1} ^{1/(\zeta-\eps)}\ge\Nulowtruncated{i+1}^{1/(\zeta-\eps)}$, and for $i\ge \iCompleteSpace$ this follows since $\Nulowtruncated{i}=\Nulowtruncated{i+1}$ and $\zeta-\eps >1$. Also note that in the case $0\le i<\iCompleteSpace$ we can simply write $\whlow{i}{\ell} = \Nuinit^{(\zeta-\eps)^{i-\ell}/(\beta-1)}$.

In order to understand the intuition behind the definition of the families $\{\Blow{i}\}$ and $\{\Bupp{i}\}$, we first observe that $\Blow i(\eps) \subseteq \Bupp i(\eps')$ for all $i\geq 0$ and all $0<\eps,\eps' <\zeta-1$. Heuristically speaking, in Section~\ref{sec:prooflower} we will show that vertices in $\Blow{i}$ of weight at least $\whlow{i}{}$ will be very likely to be infected by time $i$; while no vertices outside $\Bupp{i}$ are likely to be infected by this time as proven in Section~\ref{sec:proofupper}. The volumes of $\Blow{i}$ and $\Bupp{i}$ grow doubly exponentially in $i$ and are denote by $\Nulow{i}/n$ and $\Nulow{i}/n$, except that we need to truncate them at $1$, which is captured by $\Nulowtruncated{i}$. The vertices of largest weight in $\Blow{i}$ have weight roughly $\whlow{i}{} = \whlow{i}{0}$, and it may be helpful to imagine that these are the vertices in $\Blow{i}$ that are first infected, namely in round $i$. Afterwards, vertices of successively smaller weight become infected: in round $i+\ell$ vertices of weight at least $\whlow{i}{\ell}$ are likely infected, where $\whlow{i}{\ell}$ is decreasing in $\ell$. This intuition is also depicted in Figure~\ref{fig:bootstrap_percolation}. Of course, this picture is overly simplified and this intuition will be made rigorous in the following two sections. 

\begin{remark}\label{rem:epseta}
When our proofs involve the parameters $\eps,\eta >0$ from Definition~\ref{def:BallsEtc}, then by the notation $O(\eps), O(\eta)$ etc.~we implicitly mean that the (positive) hidden constants only depend on the parameters $d, \alpha, \beta, \wmin, \pld$, and $k$ of the model, but not on $\eps$ or $\eta$. To enhance readability, in all proofs we stick to the convention that if $\eps$ and $\eta$ occur together, then $\eta= \eta(\eps) >0$ is chosen so small that $C \eta < c \eps$ for all constants $C$ and $c$ that depend only on the model parameters. In particular, the expression $\Omega(\eps)-O(\eta)$ will always be positive in our proofs.
\end{remark}

\subsection{Lower bound on the speed}\label{sec:prooflower}

In this section we show lower bounds for the probability that a vertex in a specific region and with a specific weight will be active in some round, provided that we start in the supercritical case. Recall that the supercritical case is defined by $\infpar = \omega(\infthr)$ if the weights follow a strong power-law, and $\infpar\ge \infthr^{1-\epsilon_0}$ for some constant $\epsilon_0>0$ otherwise. The same bounds also hold in the critical case if at least $k$  ``heavy'' vertices are activated in the first round, which happens with probability $\Omega(1)$. 

The key idea is that the infection spreads in two ways: \emph{(i)} from heavy vertices (weight at least $\whlow{i}{}$) in one region ($\Blow{i}$) to heavy vertices (weight at least $\whlow{i+1}{}$) in the next region $\Blow{i+1}$, where the volume of the region increases by an exponent of (at most) $\zeta-\eps$ in each step, and \emph{(ii)} from vertices of weight at least $\whlow{i}{\ell}$ to nearby vertices of weight at least $\whlow{i}{\ell+1}=\whlow{i}{\ell}^{1/(\zeta-\eps)}$.

More formally, for any integers $i\ge0$ and $\ell\ge 0$ we abbreviate the set of all vertices in $\Blow i$ of weight at least $\whlow{i}{}$ by
$$
\heavyVertices{i}{}:=\VS{\ge \whlow{i}{}}{\Blow{i}}{}{},
$$
 and we call such vertices \emph{heavy} if the value of $i$ is clear from the context. As usual, superscripts in the notation denote active vertices, so $\heavyVertices{i}{\le j}:=\VS{\ge \whlow{i}{}}{\Blow{i}}{\le j}{}$. Furthermore we denote the event that in round $i+3$ all vertices in $\heavyVertices{i}{}$ are active by 
 $$
 \eventH{i}:=\left\{\heavyVertices{i}{}\subseteq \VS{}{}{\le i+3}{} \right\}.
 $$ 
The following theorem gives lower bounds on the probability that a vertex is active in some round.

\begin{thm}\label{thm:speedlower}
Let $\zeta$, $\eps$ and $\eta$ be as in Definition~\ref{def:BallsEtc}. Assume furthermore that we are in the supercritical case, or instead that $|\heavyVertices{0}{\leq 1}|\ge k$. Then the following is true:
\begin{enumerate}[(a)]
\item Whp, it holds that for all $i\geq 0$ the bounds $|\heavyVertices{i}{}| = \Nulowtruncated i^{\Omega(\eta)}$ and $|\heavyVertices{i}{}|=O(\Nulowtruncated{i})$ are satisfied, where the hidden constants are uniform over all $i\ge 0$.
\item Whp all the events $\eventH{i}$ occur.
\item There exist constants $C_0,C_1,C_2>0$ such that the following holds: Let $v = (\x v,\w v)$ be any vertex with fixed position and weight and let $i, \ell \geq 0$ be such that $\x v \in \Blow{i}$ and $\w v \geq \max\{\whlow{i}{\ell},C_0\}$. Then for sufficiently large $n \in \NN$,
\[
\Pr[v \in \VS{}{}{\le i+3+\ell}{} \mid \eventH{0},\ldots, \eventH{i}] \geq 1-\exp\left[-C_1\Nulowtruncated i^{{C_2(\zeta-\eps)^{-\ell}}}\right]. 
\]
\end{enumerate}
\end{thm}

The theorem agrees with the above intuition in the following sense: if $j$ is the first round in which a vertex has, say, probability $1/2$ to be active according to the bound in Theorem~\ref{thm:speedlower} (c), then $j$ agrees with the round that is predicted by the above intuition, up to additive constants. We will see in Section~\ref{sec:inftime} that Theorem~\ref{thm:speedupper} provides matching lower bounds on $j$, up to lower order terms. So in this sense, Theorem~\ref{thm:speedlower} is tight. 
\begin{rem}\label{rem:speedlower}
Our proof will in fact show that (c) still holds if we replace $\Blow i$ by an arbitrary ball of the same volume, and that it suffices if only a constant fraction of all heavy vertices is active. 

More precisely, let $B$ be any ball and restrict the process to $B$, i.e.\ vertices become infected only if they lie in $B$ and have at least $k$ infected neighbours in $B$.  Moreover, let $\eventH{i',B}$ be the event that in round $i'$ at least half of the vertices in $B$ of weight at least $(n\V B)^{1/(\beta-1+\eta)}$ are active. Then there are constants $C_0,C_1,C_2>0$ such that for any vertex $v = (\x v,\w v)$ with fixed position $\x v \in B$, and with fixed weight $\w v \geq\max\left\{ (n\V B)^{(\zeta-\eps)^{-\ell}/(\beta-1+\eta)},C_0\right\}$ we have
\[
\Pr[v \in \VS{}{}{\le i'+\ell}{} \mid \eventH{i',B}] \geq 1-\exp\left[-C_1 (n\V B)^{C_2 (\zeta-\eps)^{-\ell}}\right],
\]
in terms of the probability measure of the restricted process. 

For the sake of readability, we omit the details and prove Theorem~\ref{thm:speedlower} only in the case $B = \Blow i$. 
\end{rem}\medskip

\begin{proof}[Proof of Theorem~\ref{thm:speedlower}] 
First, let us observe that  for all $i\ge\iCompleteSpace$, i.e.\ when $\Nulow{i}\ge n$, we have $\Blow{i}=\Blow{\iCompleteSpace}=\Space$, $\heavyVertices{i}{}=\heavyVertices{\iCompleteSpace}{}$, and $\whlow{i}{\ell}=\whlow{\iCompleteSpace}{\ell}$ for all $\ell\ge 0$. Consequently, the event $\eventH{\iCompleteSpace}$ implies the event $\eventH{i}$, and we can restrict ourselves to $0\le i\le \iCompleteSpace$ when proving (a), (b), and (c).

{\it (a)} For $0\le i\le \iCompleteSpace$, by definition of $\heavyVertices{i}{}$, we have
\[
\EE[|\heavyVertices{i}{}|] \geq \Nulowtruncated i \whlow{i}{0}^{1-\beta-\eta/2} = \Nulowtruncated i^{\Omega(\eta)},
\]
where the implicit constant is independent of $i$. Because the random variable $|\heavyVertices{i}{}|$ is Poisson distributed (cf.\ Fact~\ref{fact:Poisson}), we have $\Pr[|\heavyVertices{i}{}| \leq \EE[|\heavyVertices{i}{}|]/2] = e^{-\Omega(\EE[|\heavyVertices{i}{}|])}$. We apply a union bound over all such $i$. By~\eqref{eq:Nulowtruncated} we thus see that with probability at least $1-\sum_{i=0}^{\iCompleteSpace} \exp(-\Nulowtruncated{i}^{\Omega(\eta)})=1-o(1)$ we have $|\heavyVertices{i}{}| = \Nulowtruncated i^{\Omega(\eta)}$ for all $0\le i\le \iCompleteSpace$. Similarly, the upper bound follows since $0\ll\EE[|\heavyVertices{i}{}|]\le \EE[|\VS{}{\Blow{i}}{}{}|]=\Nulowtruncated{i}$.

{\it (b)} We first show that in the supercritical case for weak power-law weights, whp $|\heavyVertices{0}{\leq 1}| \geq k$. Let $v=(\x{v},\w{v})$ be a vertex in $\heavyVertices{0}{}$. In particular $\w v \ge \whlow{0}{} = \nu^{1/(\beta-1+\eta)} = \omega(\nu^{(1-\epsilon_0)/(\beta-1)})$, where the last step follows since the assumption on $\eta$ in Definition~\ref{def:BallsEtc} implies $\beta-1+\eta < (\beta -1)/(1-\epsilon_0)$. Then we claim that in round 1, $v$ will be active with high probability. We may restrict ourselves to the case $\w v \leq \Nuinit$, since larger weights make it only easier to become active. Consider a ball around $v$ with the property that every vertex of weight at least $\wmin$ (so \emph{all} vertices) in this ball have probability $\Omega(1)$ to connect to $v$. Observe that by condition~\eqref{eq:edgeprob1} and~\eqref{eq:edgeprob2} on the edge probabilities we may choose the ball to have volume $\Omega(\w v /n)$. (For $\alpha < \infty$ we may choose the volume to be exactly $\w v/n$, for $\alpha = \infty$ we may have to choose it smaller by at most a constant factor.) Since $\w v \leq \Nuinit$, at least a constant fraction of this ball lies in $\Binit$. Hence, $\EE[|\VS{}{}{\leq 0}{v}|] = \Omega(\infpar \w v) = \omega(1)$ since in the supercritical regime $\infpar \geq \nu^{-(1-\epsilon_0)/(\beta-1)}$. Since $|\VS{}{}{\leq 0}{v}|$ is a Poisson distributed random variable (cf.\ Fact~\ref{fact:Poisson}), $v$ becomes active with high probability. In particular, if we fix any $k$ vertices in $\heavyVertices{0}{}$ then whp all $k$ of them will be active in round $1$. This implies that whp $|\heavyVertices{0}{\le1}| \geq k$, as claimed. Thus we have unified both cases.

Next we show that $|\heavyVertices{0}{\le 1}| \geq k$ implies that whp at least an $\Omega(1)$ fraction of all vertices in $\heavyVertices{0}{}$ is active in round $2$. We denote the former event by $\eventB$ and the latter event by $\eventC$. To avoid re-exposing edges, we will treat some vertices in $\heavyVertices{0}{= 0}$ separately. More precisely, if $|\heavyVertices{0}{= 0}| \leq k$, we let $X := \heavyVertices{0}{= 0}$; otherwise we choose $X$ to be some fixed subset of $\heavyVertices{0}{= 0}$ of size $k$.  So formally $X = f(\heavyVertices{0}{= 0})$, where $f\colon\heavyVertices{0}{} \to \binom{\heavyVertices{0}{}}{\leq k}$ is a fixed function. 

Now for any $Y \subseteq \heavyVertices{0}{}$ with $|Y| \leq k$ we define $\eventD(Y)$ to be the event that an $\Omega(1)$ fraction of the vertices in $\heavyVertices{0}{}$ are adjacent to all $y \in Y$ (with a sufficiently small hidden constant). Note that for all $u,v \in \heavyVertices{0}{}$, the probability for the edge $\{u,v\}$ to appear is uniformly bounded from below by $\Omega(1)$, since $\whlow{0}{}^2/\Nuinit=\omega(1)$. Therefore, $\Pr[\eventD(Y)] = 1-o(1)$ uniformly for all $Y$. Moreover, observe that both $\eventD(Y)$ and $\eventB$ are monotone increasing with respect to the edge indicators. Therefore, by the FKG inequality, Theorem~6.3.2 in~\cite{alon2004probabilistic}, we have
\begin{equation*}
\Pr\left[\eventD(Y) \cond \eventB \right] \geq  \Pr\left[\eventD(Y) \right] = 1-o(1),
\end{equation*} 
uniformly for all $Y$. Consequently, by taking the expectation over $X$, we also have $\Pr\left[\eventD(X) \cond \eventB \right] = 1-o(1)$. Thus we obtain 
\begin{align*}
\Pr\left[\eventC \cond \eventB\right] & \geq \Pr\left[\eventC, \eventD(X)\cond \eventB\right]  = (1-o(1))\Pr\left[\eventC \cond \eventD(X),\eventB\right]
\end{align*}
Therefore it suffices to prove that $\Pr\left[\eventC \cond \eventD(X), \eventB\right] = 1-o(1)$. So assume that the events $\eventD(X)$, $\eventB$ hold. Note that $\eventD(X)$ and $\eventB$ are already determined after uncovering the positions and weights of all vertices, the set $\VS{}{}{\le 0}{}$ of initially active vertices, and the edges that go out of $\VS{}{}{\le 0}{}$. So fix any outcome of these random steps so that $\eventD(X)$ and $\eventB$ hold. Note in particular that we do not need to uncover any edges between vertices in $V\setminus \VS{}{}{\le 0}{}$.

Since $\eventB$ holds, the size of $\heavyVertices{0}{=1}$ is at least $k-|X|$, and we fix any subset $X' \subseteq \heavyVertices{0}{=1}$ of size exactly $k-|X|$. Moreover, since $\eventD(X)$ holds, there is a set $Z \subseteq \heavyVertices{0}{}$ of size $\Omega(|\heavyVertices{0}{}|)$ such that all $z\in Z$ are adjacent to all $x \in X$. For every vertex $z \in Z \setminus \heavyVertices{0}{\leq 1}$, we may still uncover the edges between $X'$ and $z$. Every such edge is present with at least constant probability since $\whlow{0}{}^2/\Nuinit=\omega(1)$. Therefore, with constant probability $z$ is adjacent to \emph{all} $x \in X'$. In this case we have $z \in \heavyVertices{0}{\leq 2}$, since it is adjacent to all vertices in $X \cup X'$, and $|X \cup X'| = k$. Of course, vertices in $z \in Z \cap \heavyVertices{0}{\leq 1}$ are also in $\heavyVertices{0}{\leq 2}$, trivially. Since the coin flips for different edges are independent of each other, it follows from a Chernoff bound that whp $|\heavyVertices{0}{\le 2}|=\Omega(|Z|)=\Omega(|\heavyVertices{0}{}|)$, as claimed. Note that the only edges within $\heavyVertices{0}{}$ that we needed to uncover were edges with endpoints in $X \cup X'$. Later on, in the final argument for proving that $\eventH{0}$ holds whp, we would need to exclude edges from these vertices from our considerations. However, since $X \cup X'$ is a negligible fraction of $\heavyVertices{0}{\leq 2}$, we may (and will) suppress this subtlety. 

We next show the analogous result in the supercritical case for strong power-law weights, i.e.\ whp at least an $\Omega(1)$ fraction of $\heavyVertices{0}{}$ is active in round $2$. Recall that $\infpar= \omega( \Nuinit^{-1/(\beta-1)})$ since we are supercritical. Let $\infpar'$ be a function with the properties $\infpar' = o(\infpar)$, $\infpar' = o(1/\whlow{0}{})=o(\Nuinit^{-1/(\beta-1+\eta)})$, and $\infpar' = \omega(\Nuinit^{-1/(\beta-1)})$, and let $w' := 1/\infpar'$. Note that $\EE[|\VS{\ge w'}{\Binit}{}{}|]=\Omega((w')^{1-\beta}\Nuinit) = \omega(1)$. As for weak power-laws, for a vertex $v$ of weight at least $w'$, we consider a ball $B$ around $v$ of volume $w'/n$. In the case $\alpha <\infty$, every vertex in $B$ has probability $\Omega(\min\{\wmin w'/w',1\}) = \Omega(1)$ to connect to $v$. In the case $\alpha = \infty$, we may achieve the same by shrinking the ball $B$ by at most a constant factor. In either case, the expected number of vertices in $\VS{}{}{\leq 0}{} \cap B$ is $\infpar n\V B  = \Omega(\infpar/ \infpar') = \omega(1)$. Hence, every vertex in $\VS{\geq w'}{}{}{}\cap \Binit$ is in $\VS{}{}{\le 1}{}$ whp. By Markov's inequality, whp the number of vertices in $\VS{\geq w'}{\Binit}{}{}$ that are \emph{not} in $\VS{}{}{\le 1}{}$ is $o(\EE[|\VS{\ge w'}{\Binit}{}{}|])$. In particular, whp $|\VS{\ge w'}{}{\leq 1}{} \cap \Binit| = \omega(1)$. Finally, for any two vertices $u \in \heavyVertices{0}{}$ and $v \in\VS{\ge w'}{\Binit}{}{}$, the probability that $u$ and $v$ are adjacent is $\Omega(1)$, since $\whlow{0}{} w'/\nu \geq \whlow{0}{}^2/\nu = \omega(1)$ with room to spare. The claim now follows as before by applying a Chernoff bound.

So we have shown that in all cases whp an $\Omega(1)$ fraction of all vertices in $\heavyVertices{0}{}$ is active in round $2$, so let us assume this. To show that $\eventH{0}$ holds whp, recall that any two vertices in $\heavyVertices{0}{}$ have probability $\Omega(1)$ to be connected. Therefore, the probability that a vertex in $\heavyVertices{0}{}$ does not become active in round $3$ is at most $\Pr[\Bin{|\heavyVertices{0}{}|}{\Omega(1)}< k] = \exp[-\Omega(|\heavyVertices{0}{}|)]=o(1/|\heavyVertices{0}{}|)$ by a Chernoff bound. Hence, by the union bound whp all vertices in $\heavyVertices{0}{}$ are active in round $3$. This proves that $\eventH{0}$ holds whp.

It remains to prove that the statement holds uniformly for all $1\le i\le \iCompleteSpace$. By (a) we may assume that for all such $i$ the set $|\heavyVertices{i}{}|$ satisfies $|\heavyVertices{i}{}|=\Nulowtruncated{i}^{\Omega(\eta)}$ and $|\heavyVertices{i}{}| = O(\Nulowtruncated i)$.

We claim that any two vertices $v_{i-1} \in \heavyVertices{i-1}{}$ and $v_i \in \heavyVertices{i}{}$ with fixed position and weight form an edge with probability $\Omega(1)$. Indeed, this follows immediately since their distance is at most $(\Nulowtruncated i/n)^{1/d}$, and hence 
\[
\frac{\w{v_{i-1}}\w{v_{i}}}{\|\x{v_i}-\x{v_{i-1}}\|^dn} \geq \frac{\whlow{i-1}{}\whlow{i}{}}{\Nulowtruncated i} = \left(\Nulowtruncated{i-1}\Nulowtruncated i^{2-\beta-\eta}\right)^{1/(\beta-1+\eta)} = \omega(1),
\]
by~\eqref{eq:Nulowtruncated} as $(\zeta-\eps)^{-1}-\beta+2-\eta=\Omega(\eps)-O(\eta)$.
Therefore, the number of edges from a vertex $v_i \in \heavyVertices{i}{}$ into $\heavyVertices{i-1}{}$ is lower bounded by a binomially distributed random variable $\Bin{|\heavyVertices{i-1}{}|}{\Omega(1)}$. By the Chernoff bound, the probability that $v_i$ has less than $k$ neighbours in $\heavyVertices{i-1}{}$ is at most $\exp[-\Omega(|\heavyVertices{i-1}{}|)]=\exp[-\Nulowtruncated{i-1}^{\Omega(\eta)}]$. Using~\eqref{eq:Nulowtruncated}, a union bound over all vertices in $\heavyVertices{i}{}$ shows that  with probability at least $1-\exp[-\Nulowtruncated{i-1}^{\Omega(\eta)}]$ every vertex in $\heavyVertices{i}{}$ has at least $k$ neighbours in $\heavyVertices{i-1}{}$. A union bound over all $1\le i\le \iCompleteSpace$ shows that whp the same is still true for all such $i$ simultaneously. Hence, by a simple induction, all the events $\eventH{i}$ occur, as required.\medskip

{\it (c)} We only give the proof in the case $\alpha < \infty$, and explain in the end the changes that are necessary for $\alpha = \infty$. For $\alpha < \infty$, we prove the statement for $C_0:=(8k)^{2d/(\eps^2(\beta-2))}$, $C_1:= 4^{-d/\eps}$ and $C_2:=[\eps(\beta-2)/2](\zeta-\eps)/(\beta-1+\eta)$, where we assume that $\eps > 0$ is sufficiently small. We use induction on $\ell$. If $\w v \geq \whlow{i}{0}$ then $\eventH{i}$ implies that $v \in \VS{}{}{\le i+3}{}$, so for $\ell =0$ there is nothing to show. So let $\ell\geq 1$. Before we start with the inductive step, we will first argue that we may assume 
\begin{equation}\label{eq:weightMin}
\Nulowtruncated{i}^{(\zeta-\eps)^{-\ell+1}}\ge C_0^{\beta-1+\eta}= (8k)^{2d(\beta-1+\eta)/(\eps^2(\beta-2))} .
\end{equation}
If \eqref{eq:weightMin} is not satisfied, then $\max\{\whlow{i}{{\ell-1}},C_0\} = \max\{\whlow{i}{\ell},C_0\} = C_0$, and thus the statements of (c) for $\ell$ and $\ell-1$ are about the same set of vertices, defined by $\x v \in \Blow{i}$ and $\w v \geq C_0$. Then the statement for $\ell$ is strictly weaker than the statement for $\ell-1$, since $\VS{}{}{\le i+3+\ell}{} \supseteq \VS{}{}{\le i+3+\ell-1}{}$, and since the claimed probability is decreasing in $\ell$. In particular, the inductive step is trivial if~\eqref{eq:weightMin} is not satisfied, so we may assume~\eqref{eq:weightMin} in the following. 

Let $v$ be a vertex with position $\x v \in \Blow i$ and with weight $\w v \geq \whlow{i}{\ell}$. We claim that every vertex in distance at most $r_\ell := (\Nulowtruncated i^{(\zeta-\eps)^{-\ell+1}}/n)^{1/d}$ with weight at least $\whlow{i}{\ell-1}$ has probability $\Omega(1)$ to connect to $v$. Indeed, this follows from 
\begin{align}\nonumber
\frac{\whlow{i}{\ell-1} \whlow{i}{\ell}}{r_\ell^d n} &\stackrel{\eqref{eq:Nulowtruncated}}{\ge} \Nulowtruncated i^{[2-\beta-\eta+(\zeta-\eps)^{-1}](\zeta-\eps)^{-\ell+1}/(\beta-1+\eta)}\\\nonumber
&\ge \Nulowtruncated{i}^{[\eps(\beta-2)/2](\zeta-\eps)^{-\ell+1}/(\beta-1+\eta) }\\
&\stackrel{\eqref{eq:weightMin}}{\ge} (8k)^{d/\eps}\ge 1.\label{eq:proofspeedlower}
\end{align}
Next let us abbreviate $W_{\ell-1}(v):=\VS{\ge \whlow{i}{\ell-1}}{\Blow{i}}{}{v}$. Since $\ell \geq 1$, we have $r_\ell \leq (\Nulowtruncated i/n)^{1/d}$, which is the diameter of the ball $\Blow i$. Hence, if we consider a ball around $v$ with radius $r_\ell$, then at least a $2^{-d}$ proportion of this ball falls into $\Blow i$. Therefore we have 
\begin{align*}
\EE[|W_{\ell-1}(v)|] &=\Omega(1) 2^{-d}r_\ell^d n\whlow{i}{\ell-1}^{1-\beta+\eta} \Nulowtruncated{i}^{[\eps(\beta-2)/2](\zeta-\eps)^{-\ell+1}/(\beta-1+\eta) }  \\
&\ge c2^{-d}\Nulowtruncated i^{\left[\eps (\beta-2)/2+2\eta)\right](\zeta-\eps)^{-\ell+1}/(\beta-1+\eta) }
\end{align*}
for some constant $c>0$ and any sufficiently large $n$. Furthermore, if the constant $\eps>0$ is  sufficiently small we obtain 
\begin{align*}
\EE[|W_{\ell-1}(v)|]&\ge 2^{-d/\eps}\Nulowtruncated i^{[\eps (\beta-2)/2](\zeta-\eps)^{-\ell+1}/(\beta-1+\eta) }\stackrel{\eqref{eq:weightMin}}{\ge} (4k)^{d/\eps}\ge 8k.
\end{align*}
Recall that $|W_{\ell-1}(v)|$ is a Poisson distributed random variable (cf.\ Fact~\ref{fact:Poisson}). A Poisson distributed random variable $X$ satisfies the Chernoff bound in the form $\Pr[X \le (1-\delta)\EE[X]] \leq (e^{-\delta}/(1-\delta)^{1-\delta})^{\EE[X]}$. For $\delta = 7/8$ we have $e^{-7/8}/(1/8)^{1/8} < e^{-1/2}$, and thus $\Pr[X \le \EE[X]/8] \leq e^{-\EE[X]/2}$. Therefore, 
\begin{align*}
\Pr[|W_{\ell-1}(v)| < k] &\le  \exp(-\EE[|W_{\ell-1}(v)|]/4)\cdot \exp(-\EE[|W_{\ell-1}(v)|]/4)\\
&\le \exp(-2k)\exp\left(- 4^{-d/\eps}\Nulowtruncated i^{[\eps (\beta-2)/2](\zeta-\eps)^{-\ell+1}/(\beta-1+\eta)}\right).
\end{align*}
So assume that $|W_{\ell-1}(v)| \geq k$, and pick any $k$ vertices $v_1,\ldots, v_k \in W_{\ell-1}(v)$. By induction hypothesis, we may assume (c) for $\ell-1$ and obtain
  \begin{align*}
  \Pr[v_1 \not\in \VS{}{}{\le i+\ell+1}{}] &\leq \exp\left[-4^{-d/\eps}\Nulowtruncated i^{[\eps (\beta-2)/2](\zeta-\eps)^{-\ell+2}/(\beta-1+\eta)}\right] \\
  &\le \exp\left[-2\cdot4^{-d/\eps}\Nulowtruncated i^{[\eps (\beta-2)/2](\zeta-\eps)^{-\ell+1}/(\beta-1+\eta)}\right]\\
  &\stackrel{\eqref{eq:weightMin}}{\le} \exp(-k^d)\exp\left[-4^{-d/\eps}\Nulowtruncated i^{[\eps (\beta-2)/2](\zeta-\eps)^{-\ell+1}/(\beta-1+\eta)}\right],
  \end{align*}
where the second inequality holds since $\zeta-\eps\ge 1+\eps$ for any sufficiently small $\eps>0$.
   The same bound applies to the other $v_j$. By a simple union bound,
\[
\Pr[v_1,\ldots,v_k \in  \VS{}{}{\le i+\ell+1}{}] \geq 1- k\exp(-k^d)\exp\left[-4^{-d/\eps}\Nulowtruncated i^{[\eps (\beta-2)/2]\frac{(\zeta-\eps)^{-\ell+1}}{\beta-1+\eta}}\right].
\]
Hence,
\begin{align*}
\Pr[v \not\in  \VS{}{}{\le i+\ell+2}{}] & \leq \Pr[|W_{\ell-1}(v)| < k] + \Pr[\{v_1,\ldots,v_k\} \not \subseteq  \VS{}{}{\le i+\ell+1}{}] \\
& \leq \left(\exp(-2k) +k\exp(-k^d)\right)\exp\left[-4^{-d/\eps}\Nulowtruncated i^{[\eps (\beta-2)/2]\frac{(\zeta-\eps)^{-\ell+1}}{\beta-1+\eta}}\right]\\
&\le \exp\left[-4^{-d/\eps}\Nulowtruncated i^{[\eps (\beta-2)/2]\frac{(\zeta-\eps)^{-\ell+1}}{\beta-1+\eta}}\right]\\
&= \exp\left[-C_1\Nulowtruncated{i}^{C_2(\zeta-\eps)^{-\ell}}\right],
\end{align*}
as required.

For $\alpha = \infty$, Equation~\eqref{eq:proofspeedlower} does not imply that the corresponding vertices connect with probability $\Omega(1)$, but it suffices to decrease $r_\ell$ by at most a constant factor to ensure this property. This can be compensated by changing (for example) $C_1$. We omit the details.
\end{proof}

\subsection{Upper bound on the speed}\label{sec:proofupper}

In this section we show upper bounds for the probability that a vertex in a specific region and with a specific weight will be active in some round (Theorem~\ref{thm:speedupper} (f)). To bound this probability, we need to condition on the event that the process does not infect too many vertices in certain regions and rounds, which we show to hold with high probability in Theorem~\ref{thm:speedupper}~(e). Recall from Definition~\ref{def:BallsEtc} that $\Bupp i$ is the ball centered around $0$ of volume $\V{}(\Bupp{i}) =\min\{\Nuupp{i}(\eps)/n,1\}$. Then we define the following families of events:
\begin{itemize}
\item For all integers $i\geq 0$ 
\[
\eventE{i}:=\{\VS{}{(\Space\setminus \Bupp i)}{\le i}{}=\emptyset\},
\]
in other words, no vertex outside of $\Bupp i$ is activated by time $i$;
\item For all integers $\ell \geq 0$, all $\eps,\eta,h>0$ and $w \geq \wmin$ let
\[
\eventF{}{\ell, w}:=\eventF{\eps,\eta,h}{\ell, w}:=\left\{|\VS{\ge w}{2^\ell \Bupp{0}}{\le \ell}{}|\le h^{\ell}w^{2-\beta+\eta}\Nuupp{0}^{1-(\zeta+\eps)^{-\ell}(\beta-1)^{-1}}\right\},
\]
and 
\[
\eventF{}{\ell}:=\eventF{\eps,\eta,h}{\ell}:=\bigcap_{w'\ge \wmin}\eventF{}{\ell,w'},
\]
i.e.\ the number of vertices in $2^{\ell}\Bupp{0}$ activated by time $\ell$ is not ``too large'';
\item For all integers $j\ge 0$ and all $\eps,\eta,h>0$ set
\[
\eventG{}{j}:=\eventG{\eps,\eta,h}{j}:=\bigcap_{j'=0}^{j}(\eventE{j'}\cap\eventF{}{j'}),
\] 
in other words, $\eventG{}{j}$ summarizes that all ``good'' events up to time $j$ hold.
\end{itemize}  

With the following Theorem~\ref{thm:speedupper}, we will show that whp the events $\eventG{}{j}$ hold for all $j$. We will use induction, so we will need to give bounds on $\Pr[\eventG{}{0}]$ (part (a) and (b)) and on the conditional probabilities $\Pr[\neg \eventG{}{j} \mid \eventG{}{j-1}]$ (part (c) and (d)). If these probabilities sum up to $o(1)$, this shows that whp all $\eventG{}{j}$ hold. However, for the induction to go through we need to use a stronger statement, which is formulated in part (f). It gives for every vertex outside of $2^{\ell+1}\Bupp{i-1}$ an upper bound on the probability to be active in round $i+\ell$. 

Even though it is not relevant for the proof, let us still discuss the formula in part f. It depends on two parameters $i$ and $\ell$. The parameter $i$ determines essentially the region that we consider. Specifically, it applies to the annulus $\Bupp{i-1} \setminus 2^{\ell+1}\Bupp{i-1}$. Note that since the size of these annuli is doubly exponentially growing in $i$, the factor $2^{\ell+1}$ is rather negligible in typical cases. For $\ell = 0$, vertices are generally very unlikely to be active in round $i+\ell$, since the probability is at most $\w v \Nuupp{i}^{-\Omega(1)}$, where the second factor is doubly exponentially small in $i$. In particular, vertices of constant weight are highly unlikely to be active. However, there are (very few) vertices of weight $\approx \Nuupp{i}^{1/(\beta-1)}$ in the annulus, and for these vertices the upper bound is larger than one. So these vertices might be active. Now, for the same $i$, let us start increasing $\ell$. Then the second factor increases towards one with doubly exponential speed in $\ell$. In particular, the right hand side becomes at least $1$ for vertices of less extreme weights. For $\ell = i$, the second factor is $\nu^{-O(1)}$. Even though $\nu = \omega(1)$, it might be helpful to think about it as ``almost constant'', since it is much smaller than $\Nuupp{i}$. So at this point, even typical vertices of ``almost constant'' weights might get active. However, the theorem is not restricted to the case $\ell \leq i$, and the second factor continues to shrink beyond this point. 

Finally, we remark that up to a constant offset in the exponent ($-2$ instead of $+3$) and some technical changes ($\Bupp{i}$ instead of $\Blow{i}$), Theorem~\ref{thm:speedupper} is a rather direct complement to Theorem~\ref{thm:speedlower}, which gives an analogous lower bound to Theorem~\ref{thm:speedupper} (f). The more technical nature of Theorem~\ref{thm:speedupper} is due to the fact that the negative statement is intrinsically harder to prove. In order to show that a vertex is active, it suffices to find a single chain by which it becomes activated. In order to show that a vertex is inactive, we need to rule out \emph{all} possible chains of activation.

\begin{thm}\label{thm:speedupper}
Let $\zeta$, $\eps$, and $\eta$ be given as in Definition~\ref{def:BallsEtc}, let $h=h(n) $ be a function satisfying $h(n)= \omega(1)$, $h(n) = o(\log n)$, and $h(n) = \Nuinit^{o(1)}$. If additionally we have $\alpha > \beta-1$, then for sufficiently large $n$ we have
\begin{enumerate}[(a)]
\item $\eventE{0}$ is always satisfied;
\item $\Pr[\eventF{}{0}] \geq 1-O(h^{-1})$;
\item There exists a constant $C_\cE>0$ such that for all $i\ge 1$ we have
\[
\Pr\left[\eventE{i} \cond\eventG{}{i-1}\right] \geq 1-h^{-C_\cE i};
\]
\item There exists a constant $C_\cF>0$ such that for all $\ell\ge 1$ we have
\[
\Pr\left[\eventF{}{\ell} \cond \eventG{}{\ell-1}\right] \geq 1-h^{-C_\cF \ell};
\]
\item Whp, the events $\eventG{}{j}$ occurs for all $j\geq 0$;
\item For all $i \geq 1$ and $\ell \geq 0$, and for every fixed vertex $v=(\x{v},\w{v})$ such that $\x v \in \Space\setminus 2^{\ell+1}\Bupp{i-1}$ and $\w v \geq \wmin$ we have 
\[ \Pr\left[v \in \VS{}{}{\le i+\ell}{} \cond\eventG{}{i+\ell-1}\right] \leq \w v 2^{\ell d}\Nuupp{i}^{-(\zeta+\eps)^{-\ell-2}/(\beta-1)}.\]
\end{enumerate}
\end{thm}
\begin{proof}
First note that all statements only become easier if the edge probabilities are decreased. Hence, by Observation~\ref{obser:infinitealpha} we may restrict ourselves to the case $\alpha< \infty$, since this case dominates the case $\alpha = \infty$. 

To prove (c), (d), and (f), we will use induction on $i+\ell$, where we set $\ell=0$ and $i=0$ in (c) and (d), respectively. In particular, in order to prove (c) and (d) for $i,\ell$, we will assume statement (f) for $i', \ell'$ as long as $i'+\ell' \leq i+\ell-1$. Throughout the proof, we will assume that $n$ is sufficiently large, without mentioning this every time; for example, for the growing functions $h, \Nuupp 0 = \omega(1)$, we will use that $h$ and $\Nuupp 0$ are larger than any fixed constant (in particular $\Nuupp 0 \geq 1$) without further comment.\medskip

{\it (a)} This statement is trivial since only vertices in $\Binit \subseteq \Bupp 0$ are active at time $0$.\medskip

{\it (b)} Fix a weight $w\geq \wmin$ and note that 
\[
|\VS{\ge w}{\Bupp{0}}{\le 0}{}|\le |\VS{\ge w}{\Binit}{}{}|
\]
since initially activation only occurs within $\Binit$. Furthermore, the right-hand side is a Poisson distributed random variable (cf. Fact~\ref{fact:Poisson}) and we have 
\begin{align}\label{eq:uppperi0ell01}
\EE[|\VS{\ge w}{\Binit}{}{}|] & = O(\Nuinit w^{1-\beta+\eta}) = O(w^{2-\beta} \Nuupp{0}^{\frac{\beta-2}{\beta-1} -\Omega(\eps)} w^{-1+\eta}) = O(w^{2-\beta} \Nuupp{0}^{\frac{\beta-2}{\beta-1}-\Omega(\eps)}), 
\end{align}
where in the last step we used that $w^{-1+\eta} = O(1)$. Now, let $\bar w$ be the weight that satisfies $\bar w^{2-\beta}\Nuupp{0}^{(\beta-2)/(\beta-1)} = h$. Then by Markov's inequality, $ |\VS{\ge \bar{w}}{\Bupp{0}}{\le 0}{}|= 0$ with probability $1-O(\Nuupp{0}^{-\Omega(\eps)}h) = 1-O(h^{-1})$ since $h = \Nuupp{0}^{o(1)}$. Note that this implies (b) for all $w \geq \bar w$.

For smaller $w$, observe in \eqref{eq:uppperi0ell01} that $\Nuupp{0}^{-\Omega(\eps)}$ dominates every $O(1)$-term for sufficiently large $n$. Let $\eventFProof{0,w}$ be the event that $\left|\VS{\ge w}{\Binit}{}{}\right| \leq (2w)^{2-\beta} \Nuupp{0}^{(\beta-2)/(\beta-1)}$ and note that
\begin{align}\label{eq:upperi0ell0}
\Pr[\eventFProof{0,w}] = 1- \exp\left[-\Omega(1)(2w)^{2-\beta} \Nuupp{0}^{(\beta-2)/(\beta-1)}\right]
\end{align}
by~\eqref{eq:uppperi0ell01} and a Chernoff bound. The exponent $(\beta-2)/(\beta-1)$ in~\eqref{eq:upperi0ell0} equals the exponent $1-(\beta-1)^{-1}$ of $\Nuupp 0$ in $\mathcal F(0,w)$. Hence, if $\eventFProof{0,w'}$ holds for some $w' \geq \wmin$, then $\eventF{}{0,w}$ holds for all $w \in [w', 2w']$. Therefore, it remains to prove $\eventFProof{0,2^s\wmin}$ for all $s\in\{0,\dots, \log_2 (\bar{w}/\wmin)-1\}$. A union bound over all such $s$ using~\eqref{eq:upperi0ell0} shows that all these events hold with probability $1-\exp\{-\Omega(h)\} = 1-O(h^{-1})$. This concludes the proof of (b).\medskip

{\it (c)} Aiming for an error bound which is uniform for all $i\ge1$, in the following arguments we provide the dependence on the parameter $i$ explicitly meaning that all hidden constants of the Landau notation are independent of $i$. This is done purely for notational convenience.

We will show that with sufficiently large probability, no vertex in $\Space\setminus \Bupp i$ has a neighbour in $\Bupp{i-1}$. This will imply the statement, since we assumed $\eventG{}{i-1}$, which means in particular that all active vertices in round $i-1$ are in $\Bupp{i-1}$. 

By Lemma~\ref{lem:nolargeweights}, with probability $1-(n\V{}(\Bupp{i}))^{-\Omega(\eta)} \geq 1-h^{-\Omega(i)}$ there is no vertex $v=(\x{v},\w{v})$ such that $\x{v}\in\Space\setminus \Bupp i$ and $\w v \geq ((2\|\x v\|)^d n)^{1/(\beta-1-\eta)}$. So let $v = (\x v, \w v)$ be a vertex satisfying $\x{v}\in\Space\setminus \Bupp i$ and $\w v \leq ((2\|\x v\|)^d n)^{1/(\beta-1-\eta)}$, and note in particular that $\|\x{v}\|\ge (\Nuupp{i}/n)^{1/d}/2 \ge 2(\Nuupp{i-1}/n)^{1/d}$ since $\Nuupp{i} = \omega(\Nuupp{i-1})$. Hence, due to Markov's inequality, the probability of $v$ having a neighbour in $\Bupp{i-1}$ is at most 
\[
\EE[|\VS{}{\Bupp{i-1}}{}{v}|] =O(1) \Nuupp{i-1}\left(\frac{\w v}{\|\x v\|^dn}\right)^{\beta-1-\eta}
\] 
by Lemma~\ref{lem:neighExp}. We call any such vertex $v=(\x{v},\w{v})$ \emph{bad}, i.e.\ $v$ is bad if it satisfies $\x{v}\in \Space\setminus \Bupp i$ and $\w v \leq ((2\|\x v\|)^d n)^{1/(\beta-1-\eta)}$, and if $v$ has at least one neighbour in $\Bupp{i-1}$. Integrating over $\r{v}:=2\|\x{v}\|$ and using Lemma~\ref{lem:integral}, we can thus bound the expected number $n_{\text{bad}}$ of bad vertices by
\begin{align*}
\EE[n_\text{bad}] & =O(1) \int_{(\Nuupp{i}/n)^{1/d}}^{\infty} \hspace{-.15cm}\r{v}^{d-1}n \int_{\wmin}^{(\r{v}^d n)^{1/(\beta-1-\eta)}}\hspace{-.15cm} \w{v}^{1-\beta+2\eta} \frac{d}{d\w{v}}\Nuupp{i-1}(\r{v}^dn)^{1-\beta+\eta}\w v^{\beta-1-\eta}d\w{v} d\r{v}.
\end{align*}
Thus, since $(\r{v}^d n)^{\eta/(\beta-1-\eta)}\le (\r{v}^d n)^{\eta}$, we have 
\begin{align*}
 \EE[n_\text{bad}] &=O(1)\int_{(\Nuupp{i}/n)^{1/d}}^{\infty} \r{v}^{-1}(\r{v}^dn)^{2-\beta+2\eta}\Nuupp{i-1}d\r{v} =O(1)\Nuupp i^{2-\beta+2\eta}\Nuupp{i-1}= \Nuupp i^{2\eta - \Omega(\eps)}, 
\end{align*}
where the last step holds since $\Nuupp{i}^{2-\beta} = \Nuupp{i}^{-1/\zeta}$ and $\Nuupp{i-1} = \Nuupp{i}^{1/(\zeta+\eps)}$ by definition of $\zeta = 1/(\beta-2)$ and of $\Nuupp{i}$. Thus by Markov's inequality, with probability at least $1-\Nuupp i^{2\eta-\Omega(\eps)} \geq 1-h^{-\Omega(i)}$ there is no such vertex. Statement (c) follows.\medskip

{\it (d)} Aiming for an error bound which is uniform for all $\ell\ge1$, in the following arguments we provide the dependence on the parameter $\ell$ explicitly meaning that all hidden constants of the Landau notation are independent $\ell$. This is done purely for notational convenience.

We distinguish two cases. For $w \geq \Nuupp{0}^{(\zeta+\eps)^{-\ell}(1+3\eta)/(\beta-1)}$, we consider the upper bound 
\[
|\VS{\ge w}{2^{\ell}\Bupp{0}}{\le \ell}{}|\le |\VS{\ge w}{2^{\ell}\Bupp{0}}{}{}|.
\]
Since we have
\[
\EE[|\VS{\ge w}{2^\ell \Bupp{0}}{}{}|]  \le c_2  w^{1-\beta+\eta} 2^{\ell d}\Nuupp 0\le c_2 w^{2-\beta}2^{\ell d}\Nuupp 0^{1-(\zeta+\eps)^{-\ell}/(\beta-1)},
\]
where $c_2$ is the constant from~\eqref{eq:powerlaw}. Furthermore, since the random variable $|\VS{\ge w}{2^\ell \Bupp{0}}{}{}|$ is Poisson distributed (cf.\ Fact~\ref{fact:Poisson}), we obtain
\begin{align}\label{eq:upperi0ell11}
&\Pr \left[|\VS{\ge w}{2^\ell \Bupp{0}}{}{}| \leq 2^{\beta-1}c_2(2w)^{2-\beta}2^{\ell d}\Nuupp 0^{1-(\zeta+\eps)^{-\ell}/(\beta-1)}\right] \\ \nonumber
&\hspace{5cm}= 1- \exp\left[-\Omega(1)w^{2-\beta}2^{\ell d}\Nuupp 0^{1-(\zeta+\eps)^{-\ell}/(\beta-1)}\right]
\end{align}
by a Chernoff bound. Now note that conditioning on the high probability event $\eventG{}{\ell-1}$ can only increase the error-probabilities by at most a multiplicative factor $1/\Pr[\eventG{}{\ell-1}] \leq 2$. Moreover, similarly as in the proof of (b), it suffices to establish the bound in~\eqref{eq:upperi0ell11} only for weights of the form $2^s\wmin$ for $s\in\{0,\dots,\log_2(\bar w_\ell/\wmin)-1\}$, where $\bar{w}_\ell$ is defined by $\bar{w}_\ell^{2-\beta}2^{\ell d}\Nuupp 0^{1-(\zeta+\eps)^{-\ell}/(\beta-1)} = h^\ell$. A union bound over all such $s$ proves that $\eventF{}{\ell,w}$ holds for all $w \geq \Nuupp{0}^{(\zeta+\eps)^{-\ell}(1+3\eta)/(\beta-1)}$ with probability $1-2\exp(-\Omega(h^{\ell})) = 1- h^{-\Omega(\ell)}$.

For the second case assume that $w \leq \Nuupp{0}^{(\zeta+\eps)^{-\ell}(1+3\eta)/(\beta-1)}$. We claim that it suffices to restrict ourselves to vertices of weight at most $\hat w:=\Nuupp{0}^{1/(\beta-1)}$. More precisely, we will show that with probability at least $1-h^{-\Omega(\ell)}$, for all $w \leq \Nuupp{0}^{(\zeta+\eps)^{-\ell}(1+3\eta)/(\beta-1)}$ we have 
\begin{equation}\label{eq:ThmUppdcase2}
|U(w)|\le h^{\ell}w^{2-\beta+\eta}\Nuupp{0}^{1-(\zeta+\eps)^{-\ell}/(\beta-1)},
\end{equation}
where $U(w):=\VS{\in [w,\hat{w}]}{2^\ell\Bupp{0}}{\le \ell}{}$. Note that this suffices since by the first case there are sufficiently few other vertices active: we have seen that with probability at least $1-h^{-\Omega(\ell)}$ for any weight $w\le  \Nuupp{0}^{(\zeta+\eps)^{-\ell}(\beta-1)^{-1}(1+3\eta)} =o(\hat w)$ (for $\ell\ge1$) we have
\begin{align*}
|\VS{\ge \hat{w}}{2^\ell \Bupp{0}}{\le \ell}{}|\le h^{\ell}\hat{w}^{2-\beta+\eta}\Nuupp{0}^{1-(\zeta+\eps)^{-\ell}/(\beta-1)}\le \frac{1}{2}h^{\ell}w^{2-\beta+\eta}\Nuupp{0}^{1-(\zeta+\eps)^{-\ell}/(\beta-1)}.
\end{align*}
  
Thus we want to bound  $\EE[|U(w)|]$ by calculating the expected number of edges having one endpoint in $\VS{}{}{\le \ell-1}{}$ and the other in $\VS{\in[w,\hat{w}]}{2^\ell\Bupp{0}}{}{}$, i.e.\ we set 
\[
\edgesSet(w):=E\left(\VS{}{}{\le \ell-1}{},\VS{\in[w,\hat{w}]}{2^\ell\Bupp{0}}{}{}\right).
\]
Furthermore we observe that each edge in $\edgesSet(w)$ is also contained in at least one of the following two edge-sets: 
\[
\edgesIn(w):=E\left(\VS{}{2^{\ell+1}\Bupp{0}}{\le \ell-1}{},\VS{\in[w,\hat{w}]}{}{}{}\right),
\]
and
\[
\edgesOut(w):=E\left(\VS{}{(\Space\setminus 2^{\ell+1}\Bupp{0})}{}{},\VS{\in[w,\hat{w}]}{2^\ell\Bupp{0}}{}{}\right);
\]
then we have $\edgesSet(w)\subseteq \edgesIn(w)\cup \edgesOut(w)$. It will turn out that the bound on $|U(w)|\le|\edgesIn(w)|+|\edgesOut(w)|$ obtained this way strong enough to prove~\eqref{eq:ThmUppdcase2}.

We start by estimating $|\edgesIn(w)|$. As a preparation, we first bound $|\VS{\ge w'}{2^{\ell+1}\Bupp{0}}{\le\ell-1}{}|$, i.e.\ the number of vertices in a slightly larger region that were already active in the previous round. Since we assumed that $\eventF{}{\ell-1}$ holds, for those vertices which are also contained in the slightly smaller region $2^{\ell-1}\Bupp{0}$ we already know that
\begin{equation}\label{eq:innerball}
|\VS{\ge w'}{2^{\ell-1}\Bupp{0}}{\le\ell-1}{}|\le 2^{(\ell-1)d}(w')^{2-\beta+\eta}\Nuupp{0}^{1-(\zeta+\eps)^{-\ell+1}/(\beta-1)}.
\end{equation}
Now if $\ell = 1$, then no other vertices were active in round $\ell-1=0$ by (a). For $\ell \geq 2$, we need to examine the remaining region $2^{\ell+1}\Bupp 0 \setminus 2^{\ell-1}\Bupp 0$. Note that this area is contained in $2^{\ell}\Bupp 1$. Hence, we may apply (f) with $i'=1$ and $\ell' = \ell-2$, and thus
\begin{align}\nonumber
\EE[|\VS{\ge w'}{(2^{\ell+1}\Bupp{0}\setminus 2^{\ell-1}\Bupp{0})}{\le \ell-1}{}|] &= O(1)2^{(\ell+1)d}\Nuupp{0}(w')^{1-\beta+\eta}w'2^{(\ell-2)d}\Nuupp{1}^{-(\zeta+\eps)^{-\ell}/(\beta-1)}\\
& =O(1) h^{\ell} (w')^{2-\beta+\eta}\Nuupp 0^{1-(\zeta+\eps)^{-\ell+1}/(\beta-1)}.\label{eq:outerring}
\end{align}
Combining equations~\eqref{eq:innerball}~and~\eqref{eq:outerring} we obtain
\begin{align}\label{eq:upperd1}
\EE[|\VS{\ge w'}{2^{\ell+1}\Bupp{0}}{\le\ell-1}{}|] =O(1) \min\left\{h^{\ell}(w')^{2-\beta+\eta}\Nuupp 0^{1-\frac{(\zeta+\eps)^{-\ell+1}}{\beta-1}}, 2^{(\ell+1)d}\Nuupp 0 (w')^{1-\beta+\eta}\right\},
\end{align}
where the second term arises from dropping the condition on being active in round $i-1$. Now we denote by $\tilde w = \Theta(1)\Nuupp 0^{(\zeta+\eps)^{-\ell+1}(\beta-1)^{-1}}2^{(\ell+1)d}h^{-\ell}$ the weight for which the two expressions in~\eqref{eq:upperd1} coincide.  Recall that for any vertex $u=(\x{u},\w{u})$ of fixed weight (and independently of its position) we have $\EE[|\VS{}{}{}{u}|]=\Theta(\w{u})$ by Lemma~\ref{lem:marginalProb}. Moreover, by Lemma~\ref{lem:weightofneighbours}, the probability $q(w)$ for a random neighbour of $u$ to have weight at least $w$ is $O(w^{2-\beta+\eta/2})$, independently of $u$. Therefore we have 
\begin{align*}
&\EE[|\edgesIn(w)|]=O(q(w))\sum_{u\in \VS{\ge 0}{2^{\ell+1}\Bupp{0}}{\le \ell-1}{}}\EE[|\VS{}{}{}{u}|] \\
& \stackrel{\text{Lemma}~\ref{lem:integral}}{=}O(q(w))\int_{0}^{\infty}  \min\left\{h^{\ell}\w{u}^{2-\beta+\eta}\Nuupp 0^{1-\frac{(\zeta+\eps)^{-\ell+1}}{\beta-1}}, 2^{(\ell+1)d}\Nuupp 0 \w{u}^{1-\beta+\eta}\right\} \left(\frac{d}{d\w{u}} \w{u}\right)d\w{u}.
\end{align*}
Using Lemma~\ref{lem:evalintegral} with $\tilde{w} := 2^{(\ell+1)d}h^{-\ell}\Nuupp 0^{(\zeta+\eps)^{-\ell+1}/(\beta-1)}$ we obtain
\begin{align} \nonumber
\EE[|\edgesIn(w)|]& = O(1)w^{2-\beta+\eta/2}2^{(\ell+1)d}\Nuupp 0 \tilde{w}^{2-\beta+\eta}\\\nonumber
& = O(1)w^{2-\beta+\eta/2}2^{(\ell+1)d(3-\beta+\eta)}h^{\ell(\beta-2-\eta)}\Nuupp 0^{1-(\zeta+\eps)^{-\ell+1}(\beta-2-\eta)/(\beta-1)}\\\nonumber
& \leq w^{2-\beta+\eta/2}h^{\ell(\beta-2-2\eta)}\Nuupp 0^{1-(\zeta+\eps)^{-\ell}(1+\Omega(\eps)-O(\eta))/(\beta-1)} \\
&\leq w^{2-\beta+\eta/2}h^{\ell(\beta-2-2\eta)}\Nuupp 0^{1-(\zeta+\eps)^{-\ell}/(\beta-1)}.\label{eq:upperd2}
\end{align}

Next we turn to the edges in $\edgesOut(w)$. If $\ell=1$, then $\VS{}{(\Space\setminus2^{\ell+1}\Bupp{0})}{\le \ell-1}{}$ is empty by (a), hence also $\edgesOut(w)$. So assume $\ell \geq 2$. Fix a vertex $v=(\x{v},\w{v})$ such that $\x{v}\in 2^{\ell}\Bupp{0}$ and $\w{v}\le \hat{w}$ and denote by $\edgesOut(w,v):=\{e\in \edgesOut(w)\mid v\in e\}$ the subset of $\edgesOut(w)$ consisting of all edges incident with $v$. Now note that every edge in $\edgesOut(w,v)$ must bridge a distance of at least $\tilde{r}_{\ell} := 2^{\ell-1} (\Nuupp 0/n)^{1/d}$ and hence Lemma~\ref{lem:integral} and Lemma~\ref{lem:evalintegral} (with $\tilde w = r^d n/{\w v}$) imply
\begin{align*}
\EE\left[\left|\edgesOut(w,v)\right|\right] & = O(1)\int_{\tilde r_\ell}^{\infty}r^{d-1}n \int_{0}^\infty w_*^{1-\beta+\eta}\frac{d}{dw_*} \min\left\{\left(\frac{w_*\w v}{r^dn}\right)^\alpha,1\right\} dw_* dr \\
& = O(1)\int_{\tilde{r}_\ell}^{\infty}r^{d-1}n \left(\frac{r^d n}{\w v}\right)^{1-\beta+\eta}dr 
 = O(1)(\tilde{r}_\ell^dn)^{2-\beta+\eta}\w v^{\beta-1-\eta} \\
& = O(1)\left(\frac{2^{d\ell}\Nuupp 0}{\w v}\right)^{2-\beta+\eta}\w v
= O(1) \Nuupp 0^{-\frac{(\zeta+\eps)^{-\ell}}{\beta-1}} \w v,
\end{align*}
where the last step follows from $(\beta-2)(\beta-2-\eta)(\zeta+\eps)^2\ge 1$ since we assumed $\w{v}\le \hat{w}=\Nuupp{0}^{1/(\beta-1)}$ and $\ell\ge 2$. Hence, 
\begin{align*}
\EE[|\edgesOut(w)|] & =O(1)\Nuupp 0\int_{w}^{\infty} \w{v}^{1-\beta+\eta/2}\frac{d}{d\w{v}}\left(\Nuupp 0^{-(\zeta+\eps)^{-\ell}/(\beta-1)}\w{v}\right) d\w{v} \\
&= O(1)w^{2-\beta+\eta/2} \Nuupp 0^{1-(\zeta+\eps)^{-\ell}/(\beta-1)}.
\end{align*}
Together with~\eqref{eq:upperd2}, this shows that the expected number of vertices in $U(w)$ is also bounded by
\[
\EE[|U(w)|]\le 2w^{2-\beta+\eta/2}h^{\ell(\beta-2-2\eta)}\Nuupp 0^{1-(\zeta+\eps)^{-\ell}/(\beta-1)},
\]
and therefore, by Markov's inequality, we have
\[
\Pr\left[|U(w)| \geq \frac{1}{2}(2w)^{2-\beta+\eta} h^{\ell}\Nuupp 0^{1-(\zeta+\eps)^{-\ell}/(\beta-1)}\right] = w^{-\eta/2}h^{-\Omega(\ell)}.
\]
As in the proof of (b), we apply a union bound over all weights of the form $2^s\wmin$ parametrised by $s\in\{0,\dots,\log_2(\hat{w}/\wmin)-1\}$, and find that with probability $1-h^{-\Omega(\ell)}$, for all $w\geq \wmin$ we have 
\begin{align}\label{eq:upperd3}
|U(w)| \leq \frac12  w^{2-\beta+\eta} h^{\ell}\Nuupp 0^{1-(\zeta+\eps)^{-\ell}/(\beta-1)}, 
\end{align}
concluding the proof of (d).\medskip

{\it (e)} It follows immediately from (a)--(d) that $\Pr[\bigcap_{j' \geq 0}\eventE{j'}] = 1-h^{-\Omega(1)}$ and $\Pr[\bigcap_{j'\geq 0}\eventF{}{j'}] = 1-h^{-\Omega(1)}$, and (e) follows by a simple union bound.\medskip

{\it (f)} Fix a vertex $v=(\x{v},\w{v})$ such that $\x v \in \Space\setminus 2^{\ell+1}\Bupp{i-1}$. The statement (f) is trivial if $\w v \geq \Nuupp{i}^{(\zeta+\eps)^{-\ell-2}/(\beta-1)}$, so assume the contrary. We first estimate the number of neighbours in $2^{\ell}\Bupp{i-1}$. Observe that every such vertex has distance at least $\tilde {r}_{i-1,\ell} := 2^{\ell-1}(\Nuupp{i-1}/n)^{1/d}$ from $v$. Therefore, using $\w v \leq \Nuupp{i}^{(\zeta+\eps)^{-\ell-2}/(\beta-1)} \leq \Nuupp{i}^{(\zeta+\eps)^{-2}/(\beta-1)}$, we obtain
\begin{align*}
\EE[|\VS{}{2^{\ell}\Bupp{i-1}}{}{v}|] & = O(1)2^{\ell d} \Nuupp{i-1}\int_{0}^\infty w^{1-\beta+\eta} \frac{d}{dw} \min\left\{\left(\frac{w \w v}{2^{(\ell-1) d}\Nuupp{i-1}}\right)^\alpha,1\right\} dw \nonumber\\
&\stackrel{\text{Lemma}~\ref{lem:evalintegral}}{=} O(1)\w v \left(\frac{2^{\ell d}\Nuupp{i-1}}{\w v}\right)^{2-\beta+\eta} \nonumber\\
& = O(1)\w v2^{\ell d} \left(\Nuupp{i}^{(\zeta+\eps)^{-1}-(\zeta+\eps)^{-2}/(\beta-1)}\right)^{2-\beta+\eta},
\end{align*}
where Lemma~\ref{lem:evalintegral} with $\tilde w = 2^{(\ell-1)d}\Nuupp{i-1}/{\w v}$ is applicable because $\alpha > \beta-1$. Since $(2-\beta-\eta)[(\zeta+\eps)^{-1}-(\zeta+\eps)^{-2}/(\beta-1)]=-(\zeta+\eps)^{-2}(1+\Omega(\eps))/(\beta-1)$ we deduce
\begin{align}
\EE[|\VS{}{2^{\ell}\Bupp{i-1}}{}{v}|] & = O(1)\w v 2^{\ell d}\Nuupp{i}^{-(\zeta+\eps)^{-2}/(\beta-1)-\Omega(\eps)}\le \frac{1}{2}\w{v} 2^{\ell d}\Nuupp{i}^{-(\zeta+\eps)^{-\ell-2}/(\beta-1)}.\label{eq:uppere1}
\end{align}
In the case $\ell = 0$, this already proves the assertion since in round $i-1$ no vertex outside of $\Bupp{i-1}$ is active by $\eventE{i-1}$, and thus $\Pr[v \in \VS{}{}{=i}{}] \leq \EE[|N(v) \cap 2^{\ell}\Bupp{i-1}|]$ by Markov's inequality. 

So assume $\ell \geq 1$. Set $U_*:=|\VS{}{(\Space\setminus 2^{\ell}\Bupp{i-1})}{\le i+\ell-1}{v}|$. In this case, we can use the induction hypothesis of statement (f) for $i'=i$ and $\ell'=\ell-1$ to estimate  
\begin{align*}
&\EE\left[\left|U_*\right|\right]= \\& O(1)\int_0^{\infty} \hspace{-.2cm}r^{d-1}n \int_0^{\infty} w^{1-\beta+\eta} \frac{d}{dw} \left(\min\left\{w 2^{(\ell-1)d}\Nuupp{i}^{-\frac{(\zeta+\eps)^{-\ell-1}}{\beta-1}},1\right\} \min\left\{\left(\frac{w\w v}{r^dn}\right)^\alpha,1\right\}\right)dw dr.
\end{align*} 
To compute this integral, note that whenever the second minimum is attained by $1$, the inner integral runs either over a polynomial in $w$ with exponent $1-\beta+\eta < -1$, or over the zero function. On the other hand, whenever the second minimum is is attained by the expression $(w\w v/(r^dn))^\alpha$, then the inner integral runs over a a polynomial in $w$ with exponent larger than $-1$ (either with exponent $\alpha -\beta +\eta >-1$, or even with exponent $\alpha -\beta +\eta +1$). Therefore, by Lemma~\ref{lem:evalintegral} for $\tilde{w} = r^dn/\w v$, we obtain in all cases
\begin{align}\nonumber
\EE\left[\left|U_*\right|\right] & = O(1)\int_0^{\infty} r^{d-1}n  \left(\frac{r^dn}{\w v}\right)^{1-\beta+\eta} \min\left\{\left(\frac{r^dn}{\w v}\right)2^{(\ell-1)d}\Nuupp{i}^{-(\zeta+\eps)^{-\ell-1}/(\beta-1)},1\right\}dr.
\end{align}
Similarly, let $r_*$ be defined by $r_*^dn/\w{v}=2^{-(\ell-1)d}\Nuupp{i}^{(\zeta+\eps)^{-\ell-1}/(\beta-1)}$, then the exponent of $r$ in the antiderivative of the integrand is positive for all $r<r_*$ and negative for all $r>r_*$. Hence,
\begin{align}\label{eq:uppere2}
\EE\left[\left|U_*\right|\right] &= O(1)\w v\left(2^{(\ell-1)d} \Nuupp{i}^{-(\zeta+\eps)^{-\ell-1}/(\beta-1)}\right)^{\beta-2-\eta}\leq \frac{1}{2}\w v2^{\ell d} \Nuupp{i}^{-(\zeta+\eps)^{-\ell-2}/(\beta-1)},
\end{align} 
where the second step can be seen by a case distinction as follows. First recall $(\zeta+\eps)(\beta-2-\eta)> 1$ since $\zeta = 1/(\beta-2)$. If $\ell \ge \ell_0$ for some suitable constant $\ell_0$, then we have $O(1)2^{(\ell-1)d(\beta-2-\eta)} \leq \tfrac12 2^{\ell d}$ since $\beta-2- \eta <1$, and the inequality follows immediately. On the other hand, if $\ell < \ell_0 = O(1)$ then the term $\Nuupp{i}^{-(\zeta+\eps)^{-\ell-2}}$ is $o(1)$. Since we strictly decrease the exponent of this term, we can swallow any constant factor, and again the inequality follows. 
Together, equations~\eqref{eq:uppere1}~and~\eqref{eq:uppere2} prove the claim due to Markov's inequality.
\end{proof}

\subsection{Isolation strategies: Proof of Theorem~\ref{thm:containment}}\label{sec:containment}

In this section we prove Theorem~\ref{thm:containment}. As outlined in Section~\ref{sec:intuition}, the corollary is a rather straightforward consequence of Theorem~\ref{thm:speedupper} (e). 
\begin{proof}[Proof of Theorem~\ref{thm:containment}]
By Theorem~\ref{thm:speedupper}, whp there is no vertex outside of $\Bupp i$ which is active in round $i$. Therefore, it suffices to (permanently) remove by the end of round $i$ all edges that cross the boundary of $\Bupp i$, i.e.\ all edges in $E(\Bupp{i},\Space\setminus \Bupp{i})$. This is very similar to~\cite[Lemma 7.1 and Theorem 7.2]{bringmann2018geometric}, where the number of edges cutting a grid is considered. It does not follow directly from this lemma since the error terms in~\cite{bringmann2018geometric} are too large for our purposes. However, what \emph{does} follow directly from their proof is that among those edges that are completely contained in $2\Bupp i$, the number of edges that cross a fixed axis-parallel hyperplane is at most $\Nuupp{i}^{\max\{3-\beta, 1-1/d\}+o(1)}$. Since the boundary of $\Bupp i$ consists of a constant number of faces, this proves the bound for those edges which have both endpoints in $2\Bupp i$. 

So it remains to consider the set $\edgesBoundary{i}:=E(\Bupp i, \Space\setminus 2\Bupp i)$. Let $\eta >0$ be any constant, and let $v = (\x v, \w v)$ be a vertex such that $\x{v}\in\Space \setminus2\Bupp i$. Then by Lemma ~\ref{lem:neighExp} (in the case $\alpha > \beta-1$) the expected number of neighbours of $v$ inside of $\Bupp i$ is
\begin{align*}
\EE[|\VS{}{\Bupp i}{}{v}|]  = O(\Nuupp i) \min\left\{\left(\frac{\w{v}}{\|\x{v}\|^d n}\right)^{\beta-1-\eta},1\right\}.
\end{align*}
Note that $v$ has distance at least $r_i := \tfrac12 (\Nuupp i/n)^{1/d}$ from the origin. Thus we may use Lemma~\ref{lem:integral} and Lemma~\ref{lem:evalintegral} with $\tilde w = r^dn$ to estimate
\begin{align*}
\EE[|\edgesBoundary{i}|] & = O(1)\int_{r_i}^\infty r^{d-1}n \int_{0}^\infty \w{v}^{1-\beta+2\eta} \frac{d}{d\w{v}} \Nuupp i  \min\left\{\left(\frac{\w{v}}{r^d n}\right)^{\beta-1-\eta},1\right\} d\w{v} dr \\ 
& = O(\Nuupp i) \int_{r_i}^\infty  r^{-1}(r^dn)^{2-\beta+2\eta}dr = O(\Nuupp i^{3-\beta+2\eta}).
\end{align*} 
Since $\Nuupp i = \omega(1)$, we can deduce that $\EE[|\edgesBoundary{i}|] \leq \Nuupp i^{3-\beta+3\eta}$ for sufficiently large $n$. Since this holds for all $\eta >0$, the claim follows.
\end{proof}

\section{Infection times: Proof of Theorem~\ref{thm:inftime}}\label{sec:inftime}

In this section we prove Theorem~\ref{thm:inftime}, which gives a precise formula for the infection time of an individual vertex. As outlined in Section~\ref{sec:intuition}, Theorem~\ref{thm:inftime} is a straightforward consequence of the upper and lower bounds for the probability to be infected that are given in Theorem~\ref{thm:speedupper} and~\ref{thm:speedlower}. However, due to the rather technical nature of these theorems, the proof is still a rather tedious calculation. We distinguish several cases as in the definition of $\actExp{\x v, \w v}$, see~\eqref{eq:defofellv}. The first two cases correspond to the first line in~\eqref{eq:defofellv}, which only apply to vertices $v$ of exceptionally high weight, as they typically do not exist in GIRGs, cf.\ Section~\ref{sec:remarks}. If the maximum in~\eqref{eq:defofellv} is $0$, Case (I), then the weight is so large that $v$ has neighbours in the source region $\Blow 0$, and will become active in a constant number of rounds. If the maximum is not $0$, Case (II), then $v$ has neighbours that are untypically close to $\Blow 0$ (closer than the neighbours of other vertices in distance $|\x v|$), but not necessarily inside of $\Blow 0$. Finally, the most relevant Case (III) corresponds to the second line of~\eqref{eq:defofellv}. For the lower bound, we discriminate yet two more subcases (IIIa) and (IIIb), where (IIIb) covers the case that $v$ lies just outside (by a small factor $2^{\ell+1}$) of the boundary of~$\Blow 0$. This case requires extra treatment, since Theorem~\ref{thm:speedupper} does not apply to such vertices. Note that this complication is natural since vertices too close to the boundary may be activated via different pathways, in particular they may have direct low-weight neighbours in $\Blow 0$.

\begin{proof}[Proof of Theorem~\ref{thm:inftime}]
Let $v=(\x{v},\w{v})$ be a fixed vertex that satisfies the conditions in Theorem~\ref{thm:inftime}. We use the parameters and notation given in Definition~\ref{def:BallsEtc}.

First we remark that it suffices to show that for every sufficiently small  $\eps >0$, whp $ \act{v} = (1\pm O(\eps))\actExp{\x{v},\w{v}} \pm O(1)$, where the hidden constants are both independent of $\eps$. Then by a standard diagonalizing argument, we also have whp $\act{v} = (1\pm o(1))\actExp{\x{v},\w{v}} \pm O(1)$. We split the proof in three parts (I), (II) and (III), ``typical'' vertices are treated in (III): \medskip

{\it (I): Assume that $\w v> ((2\|\x v\|)^dn)^{1/(\beta-1)}$ and the maximum in~\eqref{eq:defofellv} is $0$, i.e.\ we also have $\w{v}\ge (2\|\x v\|)^dn/\Nuinit$.} In this case, the lower bound on $\act{v}$ is trivial, so we show the upper bound. Let $i\geq 1$ be so large that $(\zeta-\eps)^i/(\beta-1+\eta) > 1$, but observe that we may still choose $i=O(1)$. 

Assume first $(2\|\x v\|)^dn \leq \Nulowtruncated{i}$, so $\x{v} \in \Blow i$. But we recall that $\x{v}\in\Space\setminus\Blow{0}$ and so we have $(2\|\x v\|)^dn >\Nuinit$, and hence $\w v >\Nuinit^{1/(\beta-1)} = \Nulow i^{\Omega(1)}$ as $i=O(1)$. Hence, we may choose $\ell = O(1)$ such that $\w v \geq \whlow{i}{\ell}= \Nulow{i}^{(\zeta-\eps)^{-\ell}/(\beta-1+\eta)}$, and it follows directly from part (c) of Theorem~\ref{thm:speedlower} that whp $v$ is active after $i+\ell+3 = O(1)$ rounds, as required. 

On the other hand, if $(2\|\x v\|)^dn \geq \Nulowtruncated{i}$, then every vertex in $\Blow i$ has distance at most $2\|\x{v}\|$ from $v$. Recall that $\whlow{i}{}=\Nulowtruncated{i}^{1/(\beta-1+\eta)}$. Moreover, since $\|\x{v}\|\le 1/2$, we have $\Nulowtruncated{i}=\min\{\Nulow{i},n\}=\Nulow{i}$, and thus $\whlow{i}{}=\Nulow{i}^{1/(\beta-1+\eta)}$. By Theorem~\ref{thm:speedlower}~(b), after $i+3$ rounds all vertices in $\VS{\ge \whlow{i}{}}{\Blow{i}}{}{}$ are active whp, and there are $\Nulow{i}^{\Omega(\eta)}=\omega(1)$ many such vertices. Note that any such vertex has probability $\Omega(1)$ to form an edge with $v$, since $\w v \whlow{i}{}/((2\|\x{v}\|)^dn) \geq \whlow{i}{}/\Nuinit = \Omega(1)$. Therefore, whp $v$ is active in round $i+4 = O(1)$, again as required.\medskip

{\it (II): Assume that still $\w v> ((2\|\x v\|)^dn)^{1/(\beta-1)}$, but that the maximum in~\eqref{eq:defofellv} is attained by the second term, i.e.\ $(2\|\x v\|)^dn/\w{v}\ge\Nuinit$.} We need to show an upper and a lower bound on $\act{v}$. For the upper bound, choose $i\geq 0$ minimal such that 
\begin{equation}\label{eq:speedupper1}
(2\|\x v\|)^dn/\w v \leq w_i,
\end{equation} 
where we recall that $\whlow{i}{}=\Nulowtruncated{i}^{1/(\beta-1+\eta)}\le \Nuinit^{(\zeta-\eps)^i/(\beta-1+\eta)}$. Observe that this $i$ satisfies
\[
i= \log \log_\Nuinit ((2\|\x v\|)^dn/\w v) /\log(\zeta-\eps) +O(1)\leq (1+O(\eps))\actExp{\x{v},\w{v}} +O(1).
\] 
By Theorem~\ref{thm:speedlower}~(b), whp all vertices in $\VS{\ge\whlow{i}{}}{\Blow{i}}{}{}$ are active in round $i+3$, and there are $\omega(1)$ of them. As in (I), we discriminate two sub-cases. 

Either $(2\|\x v\|)^dn \geq \Nulowtruncated{i}$, which implies $\Nulowtruncated{i}=\Nulow{i}$ as before. In this case, the distance from $v$ to any point in $\Blow i$ is at most $2\|\x v\|$, and $v$ has probability $\Omega(1)$ to form an edge with each vertex in $\VS{\ge\whlow{i}{}}{\Blow{i}}{}{}$ by~\eqref{eq:speedupper1}. By Theorem~\ref{thm:speedlower}~(b), whp all these vertices are active in round $i+3$, and there are $\omega(1)$ of them, so whp $v$ will be active in round $i+4$. 

Or $(2\|\x v\|)^dn \leq \Nulowtruncated{i}$, hence $\x{v}\in \Blow{i}$. Then we observe that by the minimality of $i$ in~\eqref{eq:speedupper1} we have $\w{v}>((2\|\x v\|)^dn)^{1/(\beta-1)}\ge (\whlow{i-1}{}\w{v})^{1/(\beta-1)} $, and consequently we have $\w{v}>\whlow{i-1}{}^{1/(\beta-2)}\ge \whlow{i}{}^{1/((\beta-2)(\zeta-\eps))}$, by~\eqref{eq:Nulowtruncated}, and thus $\w{v}\ge \whlow{i}{}$, because $(\beta-2)(\zeta-\eps)<1$. Therefore, by Theorem~\ref{thm:speedlower}~(b) and~(c) whp $v$ is active in round $i+3$. In either case, whp $v$ is active in round $i+O(1)$, i.e.\ $\act{v}\le i+O(1)\le (1+O(\eps))\actExp{\x{v},\w{v}}+O(1)$, as required. \smallskip

For the lower bound, if $(2\|\x v\|)^dn/\w v \leq \Nuupp 0 = \Nuinit^{(\beta-1)/(\beta-2)+\eps}$ then $\actExp{\x{v},\w{v}} = O(1)$, so there is nothing to show. 

Otherwise, $(2\|\x v\|)^dn/\w v \geq \Nuupp 0 \geq \Nuupp 0^{1/(\beta-1-\eta)}$, so we may choose $i \geq 0$ to be maximal such that
\begin{equation}\label{eq:speedlower1}
(2\|\x v\|)^dn/\w v \geq \tilde w_i,
\end{equation} 
where $\whupp{i}{} :=\Nuupp i^{1/(\beta-1-\eta)}=\Nuupp 0^{{(\zeta+\eps)^{i}/(\beta-1-\eta)}}$. Note that this $i$ satisfies 
\[
i= \log \log_\Nuinit ((2\|\x v\|)^dn/\w v) /\log(\zeta+\eps) -O(1)\geq (1-O(\eps))\actExp{\x{v},\w{v}} -O(1).
\]
Let $\ell = O(1)$ be sufficiently large such that $\zeta^{\ell} > (\beta-1-\eta)$. If $i \leq \ell$, then $i=O(1)$, and there is nothing to show. Otherwise,~\eqref{eq:speedlower1} implies in particular
\[
(2\|\x v\|)^dn \geq \Omega(1)\Nuupp i^{1/(\beta-1-\eta)} = \Omega(1)\Nuupp{i-\ell}^{(\zeta+\eps)^{\ell}/(\beta-1-\eta)} > 2\Nuupp{i-\ell}.
\]
Hence, by Lemma~\ref{lem:neighExp} (with $C=2^{1/d}>1$) we obtain
\begin{align*}
\EE[|N(v) \cap \Bupp{i-\ell}|] & = O(1)\Nuupp{i-\ell} \left(\|\x v\|^dn/\w v\right)^{1-\beta+\eta}.
\end{align*}
Using~\eqref{eq:speedlower1}, we may continue
\begin{align*}
\EE[|\VS{}{\Bupp{i-\ell}}{}{v}|] & = O(1)\Nuupp{i-\ell}/ \Nuupp i = \Nuupp i^{-\Omega(1)},
\end{align*}
where the last step holds for any $\ell\ge 1$. By Markov's inequality, whp $v$ has no neighbours in $\Bupp{i-\ell}$. On the other hand, by Theorem~\ref{thm:speedupper} whp there is no active vertex outside of $\Bupp{i-\ell}$ in round $i-\ell$. Therefore, whp $v$ is not active in round $i-\ell+1$, i.e.\ $\act{v} > i-\ell+1 \geq (1-O(\eps))\actExp{\x{v},\w{v}} -O(1)$, as required.\medskip

{\it (III): Assume $\w v\leq ((2\|\x v\|)^dn)^{1/(\beta-1)}$.} Again we need to show an upper and a lower bound for $\act{v}$. For the \emph{upper bound}, let $i\geq 0$ be minimal with the property that $\x v \in \Blow{i}$, i.e.\
\begin{equation}\label{eq:speedupper2}
(2\|\x v\|)^dn \le \Nulowtruncated{i} = \min\{\Nuinit^{(\zeta-\eps)^i},n\}
\end{equation}
Observe that $i\ge 1$ because $n\ge(2\|\x v\|)^dn > \Nuinit$ since $\x{v}\in \Space\setminus\Blow{0}$. Therefore, we have $(2\|\x v\|)^dn \geq \Nulowtruncated{i-1}\ge  \Nulowtruncated i^{1/(\zeta-\eps)}$ by minimality of $i$ and~\eqref{eq:Nulowtruncated}. Let $\ell \geq 0$ be minimal with the property that 
\begin{equation}\label{eq:speedupper3}
\w v > ((2\|\x v\|)^dn)^{(\zeta-\eps)^{-\ell}/(\beta-1)} 
\end{equation}
Since we are in the case $\w v\leq ((2\|\x v\|)^dn)^{1/(\beta-1)}$, we have $\ell \ge 1$, and thus~\eqref{eq:speedupper3} is false if we replace $\ell$ by $\ell-1$. By minimality of $i$, the right hand side of~\eqref{eq:speedupper3} is at least $\Nulowtruncated{i}^{(\zeta-\eps)^{-\ell-1}/(\beta-1)}\ge\Nulowtruncated{i}^{(\zeta-\eps)^{-\ell-1}/(\beta-1+\eta)}=\whlow{i}{\ell+1}$ and recall that we only consider weights $\w{v}=\omega(1)$. Hence,  Theorem~\ref{thm:speedlower}~(c) applies for $\ell+1$, and, if we condition on events that hold whp, tells us that $v$ is active in round $i+\ell+4$ with probability 
\[
1-\exp\left[-C_1\Nulowtruncated{i}^{C_2(\zeta-\eps)^{-\ell-1}}\right]=1-o(1),
\]
where the last inequality holds due to the following estimate
\[
\Nulowtruncated{i}^{(\zeta-\eps)^{-\ell-1}}\stackrel{\eqref{eq:speedupper2}}{\ge}((2\|\x v\|)^dn)^{(\zeta-\eps)^{-\ell-1}}\stackrel{\eqref{eq:speedupper3}\text{ for }\ell-1}{\ge} \w{v}^{(\beta-1)/(\zeta-\eps)^2}=\omega(1).
\]
It remains to note that by the choice of $i$ and $\ell$ we have 
\[
i = \log \log_\Nuinit ((2\|\x v\|)^dn)/\log(\zeta-\eps) +O(1)
\]
 and $\ell = i + \log \log_\Nuinit \w v/\log(\zeta-\eps) +O(1)$. Hence, whp $\act{v}\le i+\ell+4=  (1+O(\eps))\actExp{\x{v},\w{v}}+O(1)$, as required.\smallskip

For the \emph{lower bound}, we distinguish yet two more sub-cases. Let $\ell \geq 0$ be the smallest non-negative integer that satisfies 
\begin{equation}\label{eq:speedlower3}
\w v^{2} \geq 2^{-(\ell+1) d}((2\|\x v\|)^dn)^{(\zeta+\eps)^{-\ell-3}/(\beta-1)}.
\end{equation}

{\it (IIIa) Assume first that $(2\|\x v\|)^dn \geq 2^{(\ell+1)d}\Nuupp 0$.} In this case, let $i\geq 1$ be maximal with the property
\begin{equation}\label{eq:speedlower4}
(2\|\x v\|)^dn \geq 2^{(\ell+1)d}\Nuupp{i-1} = 2^{(\ell+1)d}\Nuupp 0^{(\zeta+\eps)^{i-1}}.
\end{equation}
It is easy to check (e.g., by using the very generous estimate $2 < \Nuupp 0^{(\zeta+\eps)^i}$) that $i$ satisfies 
\begin{equation}\label{eq:speedlower6}
i \geq  \log \log_\Nuinit((2\|\x v\|)^dn)/\log(\zeta+\eps) -O(\log(\ell+1)).
\end{equation}
 
 If $\ell =O(1)$ then $\w v = ((2\|\x v\|)^dn)^{\Theta(1)}$ and therefore
 \[
 \actExp{\x{v},\w{v}} = \log \log_\Nuinit((2\|\x v\|)^dn)/|\log(\beta-2)| \pm O(1).
 \]
Since by Theorem~\ref{thm:speedupper} (c) whp no vertex outside of $\Bupp{i-1}$ is active in round $i-1$ and $\x{v}\not\in\Bupp{i-1}$ by~\eqref{eq:speedlower4}, it follows that whp $\act{v}>i-1=  (1-O(\eps))\actExp{\x{v},\w{v}}-O(1)$, as required. This settles the case $\ell = O(1)$.
 
Next observe that by maximality of $i$ in~\eqref{eq:speedlower4} there exists $0 \leq j = O(\log(\ell+1))$ such that $(2\|\x v\|)^dn \leq \Nuupp{i+j}$. In particular, if $\ell>C$ for some sufficiently large constant $C>0$ then $j < \ell$. Since we have already treated the case $\ell = O(1)$, we may henceforth assume that $\ell >C$. Then $\ell-j >0$, and by~\eqref{eq:speedlower4} the requirements of Theorem~\ref{thm:speedupper}~(f) are met for $i$ and $\ell - j$. Since in particular $\ell\ge 1$, by the choice of $\ell$, we have 
\begin{align}\nonumber
\w v^{2} & \leq 2^{-\ell d}((2\|\x v\|)^dn)^{(\zeta+\eps)^{-\ell-2}/(\beta-1)} \leq 2^{-\ell d}\Nuupp{i+j}^{(\zeta+\eps)^{-\ell-2}/(\beta-1)} \\
& =  2^{-\ell d}\Nuupp{i}^{(\zeta+\eps)^{-(\ell-j)-2}/(\beta-1)},\label{eq:speedlower5}
\end{align}
and therefore Theorem~\ref{thm:speedupper} (f) yields that $v$ is not active in round $i+\ell-j$ with probability at least 
\[
1-\w{v}2^{(\ell-j)d}\Nuupp{i}^{-(\zeta+\eps)^{-(\ell-j)-2}/(\beta-1)}\stackrel{\eqref{eq:speedlower5}}{\geq}1-\w{v}^{-1}2^{-j d} = 1-o(1).
\]
In order to relate $i+\ell-j$ with $\actExp{\x{v},\w{v}}$, we derive $(2\|\x v\|)^dn \geq \Nuupp 0^{(\zeta+\eps)^{i-1}}$ from~\eqref{eq:speedlower4}, and plug it into~\eqref{eq:speedlower3} to obtain 
\begin{equation*}
\w v^2 \geq 2^{-(\ell+1) d}\Nuupp 0^{(\zeta+\eps)^{i-\ell-4}/(\beta-1)} = 2^{-(\ell+1) d}\Nuinit^{\Theta(1)(\zeta+\eps)^{i-\ell}}.
\end{equation*}
Hence, taking logarithms on both sides, 
\begin{equation}\label{speedlower7}
\Theta(1)(\zeta+\eps)^{i-\ell}\log \Nuinit \leq 2\log \w{v} + (\ell+1) d\log 2  \leq 4d\max\{\log \w{v},\ell\}.
\end{equation}
If the maximum is attained by $\log \w{v}$, then~\eqref{speedlower7} gives $\ell \geq i-\log \log_\nu \w{v}/\log(\zeta+\eps) -O(1)$, and together with~\eqref{eq:speedlower6} and $j = O(\log(1+\ell))$, we conclude $i+\ell-j = (1-O(\eps))\actExp{\x{v},\w{v}} -O(1)$, as required. On the other hand, if the maximum in~\eqref{speedlower7} is attained by $\ell$, then~\eqref{speedlower7} yields
\[
\ell + \frac{\log \ell}{\log(\zeta+\eps)}= i + \frac{\log \log \Nuinit}{\log(\zeta+\eps)} -O(1) = i - \frac{\log \log_\Nuinit \w{v}}{\log(\zeta+\eps)} -O(1),
\]
where the second inequality comes from $\w{v} = \omega(1)$. Thus we obtain again $i+\ell-j = (1-O(\eps))\actExp{\x{v},\w{v}} -O(1)$, as required. This concludes the proof of the lower bound in the case $(2\|\x v\|)^dn \geq 2^{(\ell+1)d}\Nuupp 0$.\smallskip

{\it (IIIb) Assume $(2\|\x v\|)^dn \leq 2^{(\ell+1)d}\Nuupp 0$.} It remains to show the lower bound on $\act{v}$ in this case. We want to apply Theorem~\ref{thm:speedupper} (f) for $i=0$, but we need to change the definition of $\ell$ slightly. Let $\ell' \geq 0$ be the smallest non-negative integer satisfying  
\begin{equation}\label{eq:speedlower2}
\w{v}^{2} \geq 2^{-(\ell'-1) d} \Nuupp 0^{(\zeta+\eps)^{-\ell'-1}/(\beta-1)}.
\end{equation}
Similar as in (IIIa), this definition implies 
\begin{equation}\label{eq:sizeEllPrime}
\ell' +\log(\ell'+1)/\log(\zeta+\eps)\geq -\log \log_\Nuinit \w{v}/\log(\zeta+\eps) -O(1).
\end{equation}
 If $\ell' \leq 1$ then $\w v = \Nuupp 0^{\Omega(1)}$. In this case, since $(2\|\x v\|)^dn \leq 2^{(\ell+1)d}\Nuupp 0$, a sufficient condition for $\ell$ to satisfy~\eqref{eq:speedlower3} is
\[
\Nuupp 0^{\Omega(1)} \geq  2^{-(\ell+1) d}(2^{(\ell+1)d}\Nuupp 0)^{(\zeta+\eps)^{-\ell-3}/(\beta-1)}.
\]
Since this is already satisfied for some large enough constant, by the definition of $\ell$, this implies $\ell = O(1)$ and thus $(2\|\x v\|)^dn = \Nuupp 0^{O(1)}$, and the lower bound is trivial, because $\actExp{\x{v},\w{v}}=O(1)$. 

So assume instead that $\ell' >1$. Let $\ell^* := \min\{\ell'-1, \actExp{\x{v},\w{v}}\}$. Then by the assumption in Theorem~\ref{thm:inftime}, we have $\ell^* \leq \tfrac1d\log_2((2\|\x v\|)^dn/\Nuupp 0)$, and hence $(2\|\x v\|)^dn \geq 2^{\ell^*d}\Nuupp 0$. Since $\ell^* < \ell'$, the reverse of~\eqref{eq:speedlower2} holds for $\ell^*$. These two properties allow us to apply Theorem~\ref{thm:speedupper} (f) with $i=0$ and $\ell^*-1$, which tells us that $v$ is not active in round $\ell^*-1$ with probability at least 
\[
1-\w{v}2^{(\ell^*-1)d}\Nuupp{0}^{-(\zeta+\eps)^{\ell^*-1}/(\beta-1)}\ge 1-\w v^{-1} = 1-o(1).
\]
It remains to show the lower bound $\ell^*=(1-O(\eps))\actExp{\x{v},\w{v}}-O(1)$. This bound will follow from the definition of $\ell^*$ if we can show $\actExp{\x{v},\w{v}} = (1+O(\eps))\ell' + O(1)$. We recall that by definition $\actExp{\x{v},\w{v}} = (1+O(\eps))(2\log\log_{\Nuinit} ((2\|\x{}\|)^dn)-\log\log_{\Nuinit} \w{})/\log(\zeta+\eps)$, where the $O(\eps)$ error term comes from replacing $|\log(\beta-2)|$ by $\log(\zeta+\eps)$. The second term in $\actExp{\x{v},\w{v}}$ occurs in~\eqref{eq:sizeEllPrime}. Since the left hand side of~\eqref{eq:sizeEllPrime} is of the form $\ell' + o(\ell')$, we may bound the summand by $-\log\log_{\Nuinit} \w{}/\log(\zeta+\eps) = (1+o(1))\ell' + O(1)$. For the other term, since $(2\|\x v\|)^dn \leq 2^{(\ell+1)d}\Nuupp 0$, we have $\log\log_\Nuinit((2\|\x v\|)^dn) = o(\ell)+O(1)$. However, by the choice of $\ell$, we have 
\[
\w v^2 \leq 2^{-\ell d}((2\|\x v\|)^dn)^{(\zeta+\eps)^{-\ell-2}/(\beta-1-\eta)} \leq (2^{(\ell+1)d}\Nuupp 0)^{(\zeta+\eps)^{-\ell-2}/(\beta-1)},
\]
and similar as for~\eqref{eq:sizeEllPrime} it can be deduced that $\ell = O(1+|\log \log_\Nuinit \w{v}|) = O(\ell')$. Hence, we can bound the first term in $\actExp{\x{v},\w{v}}$ as $\log \log_\Nuinit((2\|\x v\|)^dn) = o(\ell') + O(1)$, which shows that $\actExp{\x{v},\w{v}} = (1+O(\eps))\ell' + O(1)$. 
This concludes the proof for the case $(2\|\x v\|)^dn \leq 2^{(\ell+1)d}\Nuupp 0$.
\end{proof}

\section{Threshold and speed of the process: Proof of Theorem~\ref{thm:main} and Theorem~\ref{thm:numberofrounds}}\label{sec:proofmain}

We prove Theorem~\ref{thm:main} and Theorem~\ref{thm:numberofrounds} together, we start by proving the second statement of Theorem~\ref{thm:numberofrounds}. 
\begin{claim}\label{claim:secondstmt}
Let $\delta>0$ be a constant. Assume that  $\alpha>\beta-1$ and that there exists a constant $C>\frac{\beta-1}{\beta-2}$ such that $\Nuinit^C \le n$, then $|\VS{}{}{\le(1-\delta)i_\infty}{}|=o(n)$ whp.
\end{claim}
\begin{proof}
We use the parameters and notation given in Definition~\ref{def:BallsEtc}, where $\eps>0$ is a sufficiently small constant. More precisely, fix $0<\eps<\min\left\{\frac{3-\beta}{\beta-2},C-\frac{\beta-1}{\beta-2},\delta \zeta |\log (\beta-2)|/2\right\}$. Recall that $\zeta=1/(\beta-2)>1$ and thus $\log(\zeta) = |\log(\beta-2)|$. Note that thus we obtain
\begin{equation}\label{eq:choiceOfEps}
\frac{|\log(\beta-2)|}{\log(\zeta+\eps)}= \frac{1}{1+\frac{\log(1+\eps/\zeta)}{|\log(\beta-2)|}}\ge 1-\frac{\eps}{\zeta|\log(\beta-2)|}> 1-\delta/2,
\end{equation}
as $\log(1+x)\le x$ and $1/(1+x)\ge 1-x$ for all $x\ge 0$. Furthermore, let $\eps_0>0$ be a constant such that $(1-\eps_0)C\ge \frac{\beta-1}{\beta-2}+\eps$, which exists since $\eps<C-\frac{\beta-1}{\beta-2}$.

  We let $i_0$ be the largest integer such that $\Nuupp{i_0-1} \leq n^{1-\eps_0}$ and note that this is well-defined and we have $i_0\ge 1$, since $$
 \Nuupp{0}=\Nuinit^{\frac{\beta-1}{\beta-2}+\eps}\le (\Nuinit^{C})^{1-\eps_0}\le n^{1-\eps_0} 
  $$
  by the choice of $\eps_0$ and using the assumption $\Nuinit^C \le n$.  Next let  $i_1$ be the smallest integer satisfying $i_1 \ge (1-\delta/2)(\log \log n)/|\log(\beta-2)|$ and observe that $2^{(i_1+1)d} = (\log n)^{O(1)}$. Therefore we have $2^{(i_1+1)d}\Nuupp{i_0-1}\le n^{1-\eps_0}(\log n)^{O(1)} = o(n)$, so whp there are $o(n)$ vertices in $2^{i_1+1}\Bupp{i_0-1}$. 

Now we consider vertices outside of $2^{i_1+1}\Bupp{i_0-1}$. First we note that by minimality of $i_1$ we have
$$
(\zeta+\eps)^{-(i_1-1)}\ge (\zeta+\eps)^{-(1-\delta/2)(\log \log n)/|\log(\beta-2)|}=(\log n)^{-(1-\delta/2)\log (\zeta+\eps)/|\log(\beta-2)|}
$$
and since $\Nuupp{i_0-1}<n^{1-\eps_0}$ we thus obtain 
$$
\Nuupp{i_0-1}^{-(\zeta+\eps)^{-(i_1-1)}}\le \exp\left(- (1-\eps_0)(\log n)^{1-(1-\delta/2)\log (\zeta+\eps)/|\log(\beta-2)|}\right)=\exp\left(-(\log n)^{\Omega(1)}\right),
$$
where the last estimate holds by~\eqref{eq:choiceOfEps}. Therefore, by Theorem~\ref{thm:speedupper} (f) each vertex in $\Space\setminus 2^{i_1+1}\Bupp{i_0-1}$ of weight at most $\log \log n$ is in $\VS{}{}{\le i_0+i_1}{}$ with probability at most 
$$
 2^{(i_1+1)d+1}\Nuupp{i_0-1}^{-(\zeta+\eps)^{-i_1-1}/(\beta-1)}\log\log n= (\log n)^{O(1)}\exp\left(-(\log n)^{\Omega(1)}\right)=o(1).
$$
 Therefore, the expected number of vertices of weight at most $\log \log n$ in $\VS{}{}{\le i_0+i_1}{}$ is $o(n)$. On the other hand, the total expected number of vertices of weight larger than $\log \log n$ is also $o(n)$. Altogether, this shows $\EE[|\VS{}{}{\le i_0+i_1}{}|]=o(n)$, and the statement follows from Markov's inequality, once we show that $i_0+i_1 \geq (1-\delta)i_\infty$.
 
To prove this, we distinguish two cases. First assume that $\Nuinit=n^{o(1)}$ and thus $\log\log_{\Nuinit} n=\omega(1)$. Recall that $\log_\Nuinit(\Nuupp{0}) = \Theta(1)$ by Definition~\ref{def:BallsEtc}. Furthermore, by the maximality of $i_0$ we have $\Nuupp{0}^{(\zeta+\eps)^{i_0}}=\Nuupp{i_0}>n^{1-\eps_0}$ implying 
 \[
 i_0 >\frac{\log \log_{\Nuupp{0}}n^{1-\eps_0}}{\log(\zeta +\eps)}\stackrel{\eqref{eq:choiceOfEps}}{\ge} \frac{1-\delta/2}{|\log(\beta-2)|} \cdot \log\left( \frac{(1-\eps)\log_\Nuinit n}{\log_\Nuinit(\Nuupp{0})}\right)
 \geq (1-\delta)\frac{\log \log_\Nuinit n}{|\log(\beta-2)|},
 \] 
 since $(1-\eps)/\log_\Nuinit(\Nuupp{0}) = \Theta(1)$ and $\log\log_{\Nuinit}n=\omega(1)$. But then clearly we have $i_0+i_1\ge (1-\delta)i_\infty$.
 
 On the other hand, if $\Nuinit\ge n^{1/C_0}$ for some constant $C_0\ge \frac{\beta-1}{\beta-2}>1$ and sufficiently large $n$, then $\log\log_{\Nuinit}n\le \log C_0$ and thus $i_\infty \le\frac{\log\log n+\log C_0}{|\log(\beta-2)|}$. Hence, we have $i_0+i_1\ge i_1\ge (1-\delta)i_\infty + R$ with $R:=\frac{\delta \log\log n-2\log C_0}{2|\log(\beta-2)|}\ge 0$ for all sufficiently large $n$, completing the proof. 
\end{proof}

\subsection*{Subcritical regime: (iii), (v).}
 We will show that whp the process does not infect any vertices in the first step. 
\begin{claim}\label{claim:sub}
$\VS{}{}{\le 1}{}=\VS{}{}{\le 0}{}$ whp.
\end{claim}
\begin{proof}
For any vertex $v=(\x{v},\w{v})$ with fixed weight and position let $\mu_v:=\EE[|\VS{}{\Binit}{}{v}|]$ denote its expected number of neighbours in $\Binit$.  We have shown in Lemma~\ref{lem:neighExp} that for any constant $C>1$, we have
\begin{equation}\label{eq:expNeigh}
\mu_v =O(\Nuinit)\cdot
\begin{cases}
\min\left\{\w v/\Nuinit, 1\right\}, &\text{if } \|\x{v}\| \le C(\Nuinit/n)^{1/d}/2,\\ 
\min\left\{\left(\w{v}/(\|\x{v}\|^d n)\right)^m,1\right\}&\text{if } \|\x{v}\| \geq C(\Nuinit/n)^{1/d}/2,
\end{cases}
\end{equation}
where $m= \min\{\alpha,\beta-1-\eta\}>1$. Since initially only vertices in $\Binit$ are activated, recall that the number $\VS{}{}{\le 0}{v}$ of \emph{initially active} neighbours of $v$ is Poisson distributed with mean $\infpar\mu_v$ (cf.\ Fact~\ref{fact:Poisson}). In particular, $\Pr[|\VS{}{}{\le 0}{v}| \geq k] = \Pr[\Poi{\infpar\mu_v} \geq k] = O(1)\min\{(\infpar\mu_v)^k,1\}$. Clearly, we can bound the number $|\VS{}{}{=1}{}|$ of vertices that turn active in round $1$ by the number of vertices that have at least $k$ neighbours in $\VS{}{}{\le 0}{}$. (It is only an upper bound since the latter also counts vertices which were already in $\VS{}{}{\le 0}{}$.)  

So let us first consider the contribution $n_{\mathrm{in}}:=|\VS{}{2\Binit}{=1}{}|$ of vertices $v=(\x{v},\w{v})$ inside of $2\Binit$. By~\eqref{eq:expNeigh} these satisfy $\mu_v = O(\w v)$, and thus by Lemma~\ref{lem:integral} and~\ref{lem:evalintegral} we obtain 
\begin{equation}\label{eq:roundsupper1}
\EE[n_{\mathrm{in}}] = O(1)\int_{0}^\infty \Nuinit w^{1-\beta+\gamma}\frac{d}{dw}\min\{(\infpar w)^k,1\} dw = O(\Nuinit \infpar^{\beta-1-\gamma}) = o(1),
\end{equation}
where $\gamma=0$ in case of a strong power-law, and otherwise $\gamma$ is an arbitrary positive constant.

On the other hand, to estimate the contribution $n_{\mathrm{out}}:=|\VS{}{(\Space\setminus 2\Binit)}{=1}{}|$ of vertices $v=(\x{v},\w{v})$ outside of $2\Binit$, we may use $\mu_v = O(\Nuinit)(\w{v}/(\|\x{v}\|^d n))^{m}$ by~\eqref{eq:expNeigh}. Furthermore, since each such vertex has distance at least $(\Nuinit/n)^{1/d}$ from the origin, Lemma~\ref{lem:integral} and Lemma~\ref{lem:evalintegral} with $\tilde w = r^d n/(\infpar\Nuinit)^{1/m}$ imply
\begin{align*}
\EE[n_{\mathrm{out}}] & = O(1)\int_{(\Nuinit/n)^{1/d}}^{\infty} r^{d-1}n \int_{0}^\infty w^{1-\beta+\eta}\frac{d}{dw}\min\left\{\left(\infpar\Nuinit\left(\frac{w}{r^dn}\right)^{m}\right)^k,1\right\} dw dr\\
& = O(1) \int_{(\Nuinit/n)^{1/d}}^{\infty} r^{d-1}n \left(\frac{r^dn}{(\infpar\Nuinit)^{1/m}}\right)^{1-\beta+\eta} dr \\
& = O(1)\Nuinit^{2-\beta+\eta}(\infpar\Nuinit)^{(\beta-1+\eta)/m}.
\end{align*}
Now we use that $\infpar = O(\Nuinit^{-1/(\beta-1)})$. Observe that this bound holds both in case (iii) and (v), and that it even holds for the critical case (ii). We derive $\infpar\Nuinit = O(\Nuinit^{(\beta-2)/(\beta-1)})$, and hence $\EE[n_{\mathrm{out}}] = \Nuinit^{-(\beta-2)(1-1/m) + O(\eta)}$. Thus, since $m>1$, if $\eta > 0$ is small enough we have 
\begin{equation}\label{eq:nout}
\EE[n_{\mathrm{out}}] = o(1).
\end{equation}
We will later use the fact that this also holds in the critical regime (ii).

Together~\eqref{eq:roundsupper1} and~\eqref{eq:nout} show that $\EE[|\VS{}{}{=1}{}|] = o(1)$, and thus by Markov's inequality whp no vertices turned active in round $1$, as claimed.
\end{proof}
 
\subsection*{Critical regime: (ii).} We first show that with constant probability no further vertices are ever activated. Observe that we are in the case of a strong power-law and here we set $\whlow{0}{}:=\Nuinit^{1/(\beta-1)}$.\footnote{In other words, for strong power-laws the quantity $w_0$ is defined without ``$+\eta$'' in the exponent.}  

\begin{claim}\label{claim:critsub}
$\VS{}{}{\le1}{} = \VS{}{}{\le0}{}$ with probability $\Omega(1)$.
\end{claim}
\begin{proof}
First observe that~\eqref{eq:nout} also holds in this regime, i.e.\ by Markov's inequality whp no vertex outside of $2\Binit$ is active in round $1$. Let $\xi>0$ be a (small) constant, to be determined later. Note that $|\VS{\ge\xi \whlow{0}{}}{2\Binit}{}{}|$ is Poisson distributed (cf.\ Fact~\ref{fact:Poisson}) with mean $O(\Nuinit (\xi \whlow{0}{})^{1-\beta}) = O(1)$, since $\xi=\Omega(1)$. Therefore the event $\mathcal{A}:=\{\VS{\ge\xi \whlow{0}{}}{2\Binit}{}{}=\emptyset\}$ occurs with probability at least $\exp[-O(1)]=\Omega(1)$. Consequently it suffices to show that $\VS{\le \xi\whlow{0}{}}{2\Binit}{\le1}{}  = \VS{}{}{\le0}{}$ with probability $\Omega(1)$ if we condition on the event $\mathcal{A}$.\smallskip

Since every vertex $v=(\x{v},\w{v})$ satisfies $\EE[|\VS{}{\Binit}{}{v}|]\le \EE[|\VS{}{}{}{v}|] = O(\w v)$, by~\eqref{eq:marginal1}, the number of neighbours in $\VS{}{}{\le 0}{}$ is dominated by a Poisson distributed random variable with mean $O(\infpar \w v)$ (cf.\ Fact~\ref{fact:Poisson}). Observe that this upper bound remains valid if we condition on the event $\mathcal{A}$, since this can only decrease the expected degree of $v$. Therefore we obtain 
\[
\Pr\left[|\VS{}{}{\le 0}{v}| \geq k\cond \mathcal{A}\right] = O(1)\min\{(\infpar\w v)^k,1\} = O((\infpar\w v)^k),
\]
and by Lemma~\ref{lem:integral} it follows that
\begin{align*}
\EE\left[|\VS{\le \xi\whlow{0}{}}{2\Binit}{=1}{}|\cond\mathcal{A}\right] & = O\left(\Nuinit (\xi\whlow{0}{})^{1-\beta}(\infpar  \xi\whlow{0}{})^k + \int_{0}^{ \xi\whlow{0}{}} \Nuinit w^{1-\beta}\frac{d}{dw} (\infpar w)^k dw\right) \\
& = O(1)\Nuinit ( \xi\whlow{0}{})^{1-\beta}(\infpar  \xi\whlow{0}{})^k = O(\xi^{k+1-\beta}),
\end{align*}
 where all the hidden constants are independent of $\xi$. Now note that we may choose $\xi>0$ small enough such that $\EE\left[|\VS{\le \xi\whlow{0}{}}{2\Binit}{=1}{}|\cond \mathcal{A}\right] \le 1/2$, and then by Markov's inequality $|\VS{\le \xi\whlow{0}{}}{2\Binit}{=1}{}|=0$ with conditional probability at least $1/2$. Thus $\VS{}{}{\le 1}{}=\VS{}{}{\le 0}{}$ with probability $\Omega(1)$, and the claim follows.
\end{proof} 

Next we show that with constant probability at least $k$ heavy vertices will be activated in the first round. Afterwards, the remaining steps will be identical with the supercritical regime, so we prove them together, cf. below.
\begin{claim}\label{claim:critsuper}
$|\VS{\ge \whlow{0}{}}{\Binit}{=1}{}|\ge k$ with probability $\Omega(1)$.
\end{claim}
\begin{proof}
We first consider the case $\alpha <\infty$. Note that by Fact~\ref{fact:Poisson}, the number of vertices in $\Binit$ of weight at least $\whlow{0}{}=\Nuinit^{1/(\beta-1)}$ is Poisson distributed with mean $\Theta(\Nuinit w_0^{1-\beta}) = \Theta(1)$. In particular, with probability $\Omega(1)$ there are at least $k$ such vertices. So assume this event holds, and let $v_1,\dots,v_k$ be $k$ distinct such vertices. For each $1\leq i\leq k$, denote by $K_i$ the intersection of $\Binit$ with the ball of volume $\Nuinit^{1/(\beta-1)}/n$ around $v_i$. Note that $n\V{K_i} = \Omega(\Nuinit^{1/(\beta-1)})$. The number of vertices in $\VS{}{K_i}{\le 0}{}$ is Poisson distributed (cf.\ Fact~\ref{fact:Poisson}) with mean $\infpar n\V{K_i}$, so in particular $\Pr[|\VS{}{K_i}{\le 0}{}| \geq k] = \Omega(1)$. Note that the events $\eventK{i} :=\{|\VS{}{K_i}{\le 0}{}| \geq k\}$ are positively associated for different $i$, i.e, conditioning on the events $\eventK{i_1}, \ldots,\eventK{i_s}$ does not decrease the probability of $\eventK{i}$ for any subset of (distinct) indices $i_1,\ldots, i_s, i$. Hence,
\[
\Pr[\forall i\in \{1,\ldots k\}:\eventK{i}] \geq \prod_{i=1}^k \Pr[\eventK{i}] = \Omega(1)
\]
by the law of conditional probability. 

Recall that, conditioned on position and weight of the vertices, the family of edge indicator random variables is independent. Now condition on the events $\eventK{i}$ and on positions and weights of $v_1,\ldots,v_k$, fix $k$ distinct vertices $u_1^{(i)},\dots,u_k^{(i)}\in \VS{}{K_i}{\le 0}{}$ for each $i$, and condition on all their positions and weights. Then for all $1 \leq i,j\leq k$ the probability for the edge $\{v_i,u_j^{(i)}\}$ edge to appear is uniformly bounded from below by $\Omega(\min\{(w_0/\Nuinit^{1/(\beta-1)})^\alpha,1\}) = \Omega(1)$. Since this is independent for all $1\leq i,j \leq k$, the conditional probability that all these $k^2$ edges appear is still $\Omega(1)$, in which case $v_i \in \VS{}{}{\le 1}{}$ for all $1\leq i \leq k$. Note that this holds uniformly for all weights and positions of the $v_i$ and $u_j^{(i)}$, so we also get $\Pr[\forall i\colon v_i \in \VS{}{}{\le 1}{} \mid \forall i: \eventK{i}] = \Omega(1)$. 

Furthermore, we observe that the number of initially infected vertices in $\Binit$ of weight at least $w_0$ is Poisson distributed (cf.\ Fact~\ref{fact:Poisson}) with mean $\Theta(\infpar \Nuinit w_0^{1-\beta})=o(1)$, and therefore the probability that there is such a vertex is $o(1)$. Consequently, the claim follows by taking a union bound in the case of $\alpha>\infty$.

The case $\alpha=\infty$ is completely analogous, except that it may be necessary to shrink the balls around $v_1,\ldots,v_k$ by at most a constant factor, so that still every vertex in the $i$-th ball has probability $\Omega(1)$ to connect to $v_i$. Since this only decreases the expected number of (active) vertices in each ball by constant factors, the remaining proof stays the same. We omit the details.
\end{proof}
 
\subsection*{Supercritical regime: (i), (iv).}\label{sec:super}
In this proof we also include the critical regime (ii), provided that at least $k$ heavy vertices got activated in the first round, i.e.\  $|\VS{\ge \whlow{0}{}}{\Binit}{=1}{}|\ge k$, where as before $\whlow{0}{}=\Nuinit^{1/(\beta-1)}$. 
\begin{claim}\label{claim:super}
Let $\delta>0$ be a constant, then we have $|\VS{}{}{\le (1+\delta) i_\infty}{}|=\Omega(n)$ whp.
\end{claim}
\begin{proof}
We use the parameters and notation given in Definition~\ref{def:BallsEtc}, where $\eps>0$ is sufficiently small, i.e.\ $0<\eps<\min\left\{\frac{3-\beta}{\beta-2},\delta \right\}$. Recall that $\iCompleteSpace$ denotes the smallest index $i \geq 0$ such that $\Nulow{i}\geq n$ (and thus $\Blow \iCompleteSpace = \Space$), and note that $\iCompleteSpace \leq (1+\eps/2)(\log \log_\Nuinit n)/|\log(\beta-2)|$ if $n$ is sufficiently large. Then there exists $\ell \leq (1+\eps/2)(\log \log n)/|\log(\beta-2)|$ such that $\whlow{\iCompleteSpace}{\ell} = \Nulowtruncated{\iCompleteSpace}^{(\zeta-\eps)^{-\ell}/(\beta-1+\eta)}= O(1)$. This means that Theorem~\ref{thm:speedlower} is in particular applicable for all vertices of weight $\omega(1)$ with parameters $\iCompleteSpace$ and $\ell$. More precisely, let $h= h(n) = \omega(1)$ be a function with $\log \log h = o(\log \log n)$ and set $j:= (1+\eps/2) i_\infty$. Then every vertex of weight at least $\wslow{h}:=h^{1/(\beta-1+\eta)}$ has probability $1-h^{-\Omega(1)}$ to be in $\VS{}{}{\le j}{}$. 

Now decompose the torus $\Space$ into balls $Q_1,\ldots, Q_s$ of volume $\Theta(h/n)$, where $s=\Theta(n/h)$.\footnote{This is possible since we use the $\infty$-norm. It would also suffice to consider any disjoint balls with total volume $\Omega(1)$.} Fix any such ball $Q$, and call $Q$ \emph{good} if in round $j$ at least half of the vertices in $\VS{\ge \wslow{h}}{Q}{}{}$ are active, and \emph{bad} otherwise. Since in expectation only a $o(1)$ fraction of the vertices in $\VS{\ge \wslow{h}}{Q}{}{}$ are inactive in round $j$, by Markov's inequality the probability that $Q$ is bad is $o(1)$. So in expectation only a $o(1)$ fraction of the sets $Q_1,\ldots,Q_s$ are bad, and again by Markov's inequality, whp at least half of them are good. 

Assume $Q$ is good. For the upcoming steps we consider the process as restricted within $Q$ (as mentioned in Remark~\ref{rem:speedlower}) and write $X_Q=X_Q(C,\ell')$ for the number of vertices of weight at least $C>0$ which become infected within the next $\ell'\ge 0$ additional steps. For some suitably chosen $\ell'=o(\log\log h)$ and sufficiently large constant $C>0$ it follows from Remark~\ref{rem:speedlower} that $\EE[X_Q]\ge \frac{2}{3}|\VS{\ge C}{Q}{}{}|$. Thus by 
Markov's inequality we have 
$\Pr\left[X_Q \geq \EE\left[\left|\VS{\ge C}{Q}{}{}\right|\right]/2\right] = \Omega(1).$

Because the restricted processes for $Q_\sigma$, $1\le \sigma \le s$, are independent, by a Chernoff bound, whp an $\Omega(1)$ fraction of the balls $Q_\sigma$, $1\le \sigma\le s$, satisfy $X_{Q_\sigma}\ge \EE\left[\left|\VS{\ge C}{Q_\sigma}{}{}\right|\right]/2=\Omega(h)$. Because $|\VS{\ge C}{Q_\sigma}{\le j+\ell'}{}|\ge X_{Q_\sigma}$, we have  whp $|\VS{}{}{\le j+\ell'}{}| = \Omega(s\cdot h) = \Omega(n)$. Since $j+\ell' \le (1+\eps) i_\infty\le (1+\delta) i_\infty$ for sufficiently large $n$, the claim follows. 
\end{proof}

\begin{proof}[Proof of Theorem~\ref{thm:main} and Theorem~\ref{thm:numberofrounds}]
Theorem~\ref{thm:main} is an immediate consequence of Claims~\ref{claim:sub},~\ref{claim:critsub},~\ref{claim:critsuper}, and~\ref{claim:super}, while Theorem~\ref{thm:numberofrounds} is proven by Claims~\ref{claim:secondstmt} and~\ref{claim:super}.
\end{proof}

\section{Concluding remarks}

We have shown that in the GIRG model for scale-free networks with underlying geometry, even a small region can cause an infection that spreads through a linear part of the population. We have analysed the process in great detail, and we have determined its metastability threshold, its speed, and the time at which individual vertices becomes infected. Moreover, we have shown how a policy-maker can utilise this knowledge to enforce a successful quarantine strategy. We want to emphasize that the latter result is only a proof of concept, intended to illustrate the possibilities that come from a thorough understanding of the role of the underlying geometry in infection processes. In particular, we want to remind the reader that bootstrap percolation is not a perfect model for viral infections (though it has been used to this end), but is more adequate for processes in which the probability of transmission grows more than proportional if more than one neighbours is active, like believes spreading through a social network (``What I tell you three times is true.''), or action potential spreading through a neuronal network.

Therefore, this paper is only a first step. There are many other models for the spread of an infection, most notably SIR and SIRS models for epidemiological applications, and we have much yet to learn from analysing these models in geometric power-law networks like GIRGs.  From a technical point of view, it is unsatisfactory that our analysis does not include the case $\alpha \leq \beta-1$. We believe that also in this case, the bootstrap percolation process is essentially governed by the geometry of the underlying space, only in a more complex way. Understanding this case would probably also add to our toolbox for analysing less ``clear-cut'' processes. 

\bibliographystyle{amsplain}
\bibliography{../../Koch_Lengler_GIRGBootstrapPercolation}

\end{document}